\documentclass[a4paper] {article}[12pt]

\usepackage{amsmath}

\usepackage[top=2 cm, bottom=2 cm, left=2.5 cm, right= 2.5 cm]{geometry}

\makeatletter
\renewcommand*\l@section{\@dottedtocline{1}{1.5em}{2.3em}}
\makeatother

\usepackage{amsfonts}
\usepackage{amssymb}
\usepackage[T1]{fontenc}

\usepackage{tikz}
\usetikzlibrary{calc}

\usepackage{CJK}
\usepackage{amsmath}
 
\usepackage{amsfonts}
\usepackage{amssymb}
\usepackage{amsthm}
\usepackage{amssymb}
\usepackage{enumerate}
\usepackage[calc]{picture}
\usepackage[all,cmtip]{xy}

\usepackage[mathscr]{eucal}
\usepackage{eqlist}

\usepackage{color}
\usepackage{abstract} 
\usepackage[T1]{fontenc}
 
\theoremstyle{plain}
\newtheorem{theorem}{Theorem}
\newtheorem{proposition}[theorem]{Proposition}
\newtheorem{lemma}[theorem]{Lemma}

\newtheorem{example}[theorem]{Example}
\newtheorem{corollary}[theorem]{Corollary}

\theoremstyle{definition}
\newtheorem{definition}{Definition}

\usepackage{etoolbox}
\newtheoremstyle{myrem}
 {3pt}
 {3pt}
 {\normalsize}
 { }
 {\itshape}
 {:}
 { }
 {}

 \theoremstyle{myrem}
 \newtheorem{remark}{Remark}
 \appto\remark{\leftskip\parindent}
 \appto\remark{\rightskip\parindent}

\usepackage{amsmath}

\numberwithin{equation}{section}
\numberwithin{theorem}{section}

\begin{document}

\begin{center}
~~~
\bigskip
~~~\\
{\Large{\textbf{Persistent  bundles  over   configuration  spaces 
and  obstructions  for  regular  embeddings  
 }}}

 \vspace{0.58cm}
 
 Shiquan Ren 

\bigskip

\bigskip

 \parbox{24cc}{{\small

{\textbf{Abstract}.}  
We  construct  persistent    bundles  over  
   configuration  spaces  of  hard  spheres  and
     use  the  characteristic  classes  of  these persistent  bundles  
 to  give  obstructions  for    
  embedding   problems.  
 The  configuration  spaces  of  $k$-hard  spheres  ${\rm  Conf}_k(X,r)$,  
 $r\geq  0$,   
 give  a  $\Sigma_k$-equivariant   filtration  of  the  configuration  space  of  $k$-points
 ${\rm  Conf}_k(X)$.  
 The  filtered  covering  map  
 from  ${\rm  Conf}_k(X,-)$  to   ${\rm  Conf}_k(X,-)/\Sigma_k$
    gives  a  canonical  persistent     bundle   $\boldsymbol{\xi}(X,k,-)$.     
 We  use  the  Stiefel-Whitney  class  of  $\boldsymbol{\xi}(X,k,-)$,  which  is  
  in  the  mod  $2$  persistent  cohomology  ring  of  ${\rm  Conf}_k(X,-)/\Sigma_k$, 
  to  give  obstructions  for  $(k,r)$-regular  embeddings  
   and  use  the  Chern class  of  $\boldsymbol{\xi}(X,k,-)\otimes \mathbb{C}$,  which  is  
  in  the  integral  persistent  cohomology  ring  of  ${\rm  Conf}_k(X,-)/\Sigma_k$, 
  to  give  obstructions  for  complex  $(k,r)$-regular  embeddings.   
  As  applications,  we   discuss the  geometric  realizations  of   the independence  complexes    
 given  by  the  regular  embeddings. 
 With  the  help  of  the  persistent  homology tools, 
 the  $k$-regular  embedding problems  of  manifolds, 
 the sphere-packing  problems  on  manifolds  and  the  
 geometric  realization  problems  of  the  independence  complexes  of  graphs  are  
 prospectively  to  be  
 computed  approximately.  
  }

\bigskip

}

\begin{quote}

\end{quote}

\end{center}

\vspace{1cc}

\section{Introduction}

      Configuration  spaces  
    attract  attention  in    various   areas  of  mathematics.        
    For  instance,     
    the  configuration  spaces  of  hard  spheres 
    (for  example,  \cite{confcomp,imrn1,phys})  is  a  geometric  model  of  
        the  sphere  packing  problem   and  have  applications  in 
        statistical  mechanics
         (for  example,   \cite{phys-hl});   
           the  topology  of  configuration  spaces  has  intrinsic     
        connections  with  iterated  loop  spaces 
          (for  example,   \cite{cohen1978,bundle1978,salvatore})
        and  braid  groups     (for  example,   \cite{cohen1978, cohen2010});   
         etc.

Let $X$ be a   CW-complex.  Let  $k$  be  a positive integer.  The 
ordered  configuration space 
${\rm  Conf}_k(X)$ is the subspace of the Cartesian product $X^k$ consisting of the  points 
 $(x_1,x_2,\cdots,x_k)$ such that $x_i\neq x_j$ if $i\neq j$. 
The symmetric group $\Sigma_k$  on $k$-letters  acts on ${\rm  Conf}_k(X)$  from the  left  by 
\begin{eqnarray}\label{eq-25-01-31-1}
\sigma(x_1,x_2,\cdots,x_k)=(x_{\sigma(1)},x_{\sigma(2)},\cdots,x_{\sigma(k)}) ,\text{\ \ \  } \sigma\in \Sigma_k.  
\end{eqnarray}
  This  induces a covering map from ${\rm  Conf}_k(X)$ to ${\rm  Conf}_k(X)/\Sigma_k$.  
The associated vector bundle of this covering map is 
\begin{eqnarray}\label{01292016-e1}
\xi(X,k): \mathbb{R}^k\longrightarrow  {\rm  Conf}_k(X)\times_{\Sigma_k}\mathbb{R}^k\longrightarrow {\rm  Conf}_k(X)/\Sigma_k   
\end{eqnarray}
where $\Sigma_k$ acts  on $\mathbb{R}^k$   from the right by 
 \begin{eqnarray}\label{eq-25-01-31-2}
(a_1,a_2,\cdots,a_k)\sigma=(a_{\sigma^{-1}(1)},a_{\sigma^{-1}(2)},\cdots,a_{\sigma^{-1}(k)}) ,\text{\ \ \  } \sigma\in \Sigma_k.   
\end{eqnarray}
The  complexification  of  $\xi(X,k)$  is  a  complex  vector  bundle 
\begin{eqnarray}\label{01292016-e2}
\xi(X,k)\otimes \mathbb{C}: \mathbb{C}^k\longrightarrow  {\rm  Conf}_k(X)\times_{\Sigma_k}\mathbb{C}^k\longrightarrow {\rm  Conf}_k(X)/\Sigma_k.    
\end{eqnarray}
Let  $\mathbb{F}$  be  the  real  numbers  $\mathbb{R}$  or  the  complex  
           numbers  $\mathbb{C}$. 
We  write  (\ref{01292016-e1})  and  (\ref{01292016-e2})  uniformly  as  
$\xi(X,k;\mathbb{F})$  where  $\mathbb{F}=\mathbb{R}$  gives    (\ref{01292016-e1}) 
and  $\mathbb{F}=\mathbb{C}$  gives   (\ref{01292016-e2}).

The     bundles    $\xi(X,k;\mathbb{F})$  
have  various  applications  in    topology  and  geometry   and  
have been extensively studied   since  1970's  (for  example,  \cite{high1,high2,homol, bundle1983,bundle1989,cohen1,bundle1978, swy}).  
 Among  these  applications,  
 the  regular  embedding  problem  (cf.  \cite{high1,high2,chi,cohen1,handel2,1996,handel3,reg2024})   is  interesting  and  attracts 
   attention  during the  last  half century.

           The   regular  embedding   problem  is   initially  studied  
           by  K.  Borsuk \cite{Borsuk}   in  1957  and  later  developed  
           by  F.R. Cohen and D. Handel  \cite{cohen1}  in  1978,  
           M.E. Chisholm \cite{chi} in 1979,  
           D. Handel and J. Segal \cite{handel3}  in  1980,  
           D.  Handel  \cite{1996}  in  1996,  
           P. Blagojevi\'{c},   W. L{\"u}ck and G. Ziegler \cite{high1} in 2016,  
           P. Blagojevi\'{c},  F.R. Cohen, W. L{\"u}ck and G. Ziegler  \cite{high2} in 2016,  
           etc.  
           A  map   $f:  X\longrightarrow  \mathbb{F}^N$  is  called  
           {\it  (real  or  complex)  $k$-regular}    
             if   for  any  distinct  $k$  points in  $X$,  
           their  images  are  linearly  independent  in  $\mathbb{F}^N$.
           When  $k\geq  2$,  such a  map  is  an  embedding.  
           By  the  Haar-Kolmogorov-Rubinstein  Theorem  (cf. \cite{handel4} and \cite[pp. 237-242]{singer}), 
           one  of  the  motivations    to  study  $k$-regular  maps  is 
            the   \v{C}eby\v{s}ev approximation.

     The    existence  problems  of   $k$-regular maps  
     as  well  as    their   obstructions    
     are  investigated.   
     For  example,    D.  Handel    studied  
      obstructions  for  $3$-regular  maps      \cite{handel1}  
     in  1979  and    studied  
        some   existence  and   non-existence  problems   of  
    $k$-regular  maps     \cite{handel2}   in  1980.

 The   bundles    $\xi(X,k;\mathbb{F})$
  can  be     
   applied  to  give  obstructions  for   the  existence  of  $k$-regular  maps. 
   Some  obstructions  using  the  Stiefel-Whitney  classes  of $\xi(X,k)$  
     for  the  existence  of  $k$-regular  maps  were  proved  by  
   F. Cohen and D. Handel  \cite{cohen1},  restated  in    \cite[Lemma~4.2]{topapp}.  
  Some  obstructions  using  the  Chern  classes  of  $\xi(X,k;\mathbb{C})$  
     for  the  existence  of   complex  $k$-regular  maps  were  proved  by  
   P. Blagojevi\'{c},  F. Cohen, W. L{\"u}ck and G. Ziegler   \cite{high2},    
   restated   in  \cite[Lemma~4.3]{topapp}.

In  this  paper,  we  apply  the  method  of  persistent  homology  to    
the  bundles  over  configuration  spaces.    
Then  we  weaken  the  $k$-regularity  to    define  a  $(k,r)$-regularity  
condition  for  embeddings  (cf.  Definition~\ref{def-7.1})
 and   
give  obstructions  for  the  $(k,r)$-regular embedding  problem.

Precisely,  
we  give  a  filtration   
  of  ${\rm  Conf}_k(X)$  by  the  configuration  spaces  of  hard spheres  
   and  construct  a  family  of  pull-back 
bundles  of    $\xi(X,k; \mathbb{F})$. 
As  the  persistent  version   of  $\xi(X,k; \mathbb{F})$,  
 we  define  the  persistent  bundles  
     $ \boldsymbol{\xi}(X,k,-; \mathbb{F})$    (cf.  eq.  (\ref{eq-3.96})).   
We  generalize  the  $k$-regularity and  define  the  $(k,r)$-regular  embeddings
 for  any  $r\geq  0$.  
 We  equip a  metric  $d$  on  $X$  and  call  a  map  
 $f:  X\longrightarrow  \mathbb{F}^N$  to  be  {\it  $(k,r)$-regular}  
 if  for  any  $k$   points   in  $X$  whose  pairwise  distances  are  greater than    $2r$,  
 their images are  linearly  independent   (cf.  Definition~\ref{def-7.1}).  
 We  prove  some   obstructions  for  the  existence  of   the   $(k,r)$-regular  maps 
by  using  the  characteristic  
classes  of      $ \boldsymbol{\xi}(X,k,-; \mathbb{F})$,
   which  are  elements  in   the  persistent  cohomology  rings
   of  filtrations  of  the  unordered  configuration  space  ${\rm  Conf}_k(X)/\Sigma_k$.

 Let  $(X,d)$  be  a  metric  space.  For  any  positive  integer  $k$
 and  any  nonnegative  number  $r$,  
the  {\it  ordered  configuration  space  of  hard spheres of  radius  $r$},
denoted  by  ${\rm  Conf}_k(X,r)$,     
is  the  subspace  of  $X^k$  consisting of the  points 
 $(x_1,x_2,\cdots,x_k)$   whose  coordinates  satisfy $d(x_i,x_j)>2r$  for  any  $i\neq  j$. 
In  particular,  the  space  ${\rm  Conf}_k(X,0)$ is ${\rm  Conf}_k(X)$.   
   The  symmetric  group  
  $\Sigma_k$  acts  on ${\rm  Conf}_k(X,r)$  
  freely  and  properly  discontinuously. 
  The  orbit  space  ${\rm  Conf}_k(X,r)/\Sigma_k$ 
   is called   the    {\it  unordered  configuration  space  
    of  hard spheres  of  radius  $r$}.  
        The  family  ${\rm  Conf}_k(X,r)$  for  all  $r\geq  0$,  denoted by  ${\rm  Conf}_k(X,-)$,  
  gives  a  filtration  of  ${\rm  Conf}_k(X)$.   
    The   $\Sigma_k$-action  on    ${\rm  Conf}_k(X)$  
     is  equivariant  with  respect  to  the  filtration  thus  it  induces
      a  filtered   orbit  space  
    ${\rm  Conf}_k(X,-)/\Sigma_k$    and  a   covering  map  of  filtered  spaces 
    from  ${\rm  Conf}_k(X,-)$  to  ${\rm  Conf}_k(X,-)/\Sigma_k$.

   Let   ${\rm  Conf}_\bullet(X,-)$  be  the  collection  of   ${\rm  Conf}_k(X,-)$   and   
   let  ${\rm  Conf}_\bullet(X,-)/\Sigma_\bullet$    be  the  collection  of  
   ${\rm  Conf}_k(X,-)/\Sigma_k$  for  $k\geq  1$.  
       A  persistent  double  complex    for  
       ${\rm  Conf}_\bullet(X,-)$ 
     will  be  constructed  in  Theorem~\ref{th-25-5.1.1}.  
    In  particular,  if   we  let   $X$      be   a  Riemannian  manifold  $M$,  
    then  the  differential  forms  on   ${\rm  Conf}_\bullet(M,-)$  
   will  give  a  persistent  double  complex.   The  $\Sigma_\bullet$-invariant  forms  
   on   ${\rm  Conf}_\bullet(M,-)$,  or  alternatively  the  differential  forms  on    
   ${\rm  Conf}_\bullet(M,-)/\Sigma_k$,  
   will  give  a  persistent  double  complex  as  well  (cf.   Subsection~\ref{ss3.2}).   
      Moreover,  if  we  let  $X$  be  a  graph  $G$  and  let  
      $d$  be   the  canonical  geodesic distance  of   $G$,  
      then  ${\rm  Conf}_\bullet(V_G,-)/\Sigma_\bullet$    will  
      give a  filtration  of  the  independence  complex  of  $G$ 
      (cf.   Subsection~\ref{ss3.1}).

        Embed    the  discrete  fibre  of  the   covering  map  
         from  ${\rm  Conf}_k(X,-)$  to  ${\rm  Conf}_k(X,-)/\Sigma_k$
        into   $\mathbb{F}^k$   as  the  unit  vectors  of  the  $k$  axes.  
       We  obtain a    persistent  vector  bundle   $\boldsymbol{\xi}(X,k,-;\mathbb{F})$  
    (cf.  Definition~\ref{def-25-01-3}  and  Theorem~\ref{pr-3.2aaz})  
    over  the  filtered  space  ${\rm  Conf}_k(X,-)/\Sigma_k$.  
    The  Stiefel-Whitney  class  of  $\boldsymbol{\xi}(X,k,-;\mathbb{R})$  
    is  a  persistent  cohomology  class  in  the   persistent cohomology  ring 
    $H^*({\rm  Conf}_k(X,-)/\Sigma_k;\mathbb{Z}_2)$  and  the  
    Chern  class  of  $\boldsymbol{\xi}(X,k,-;\mathbb{C})$  
    is  a  persistent  cohomology  class   in   the  persistent cohomology  ring 
    $H^*({\rm  Conf}_k(X,-)/\Sigma_k;\mathbb{Z})$.  
        The  next  theorem  will  be  proved  in  
    Theorem~\ref{25-cor1}   and  Theorem~\ref{25-cor2}.  
    
    \begin{theorem}[Main  Result]
    \label{th-main-25-01-29}
    \begin{enumerate}[(1)]
    \item
   If
 $f: X\longrightarrow \mathbb{R}^N$ is  
 a $(k,r)$-regular map,   
 then  
 \begin{eqnarray}\label{eq-3.1x-introd-xx}
 N\geq  \sup  \{ t(k)  \mid    
 \bar w_{t(k)}( \boldsymbol{\xi}(X,k,-+r;\mathbb{R}))\neq 0  \} +k;  
 \end{eqnarray}
    \item
   If
 $f: X\longrightarrow \mathbb{C}^N$  is  
 a  complex   $(k,r)$-regular map,   
then   
 \begin{eqnarray}\label{eq-3.1y-introd-yy}
 N\geq   \sup\{ t(k)\mid   \bar  c_{t(k)}( \boldsymbol{\xi}(X,k,-+r;\mathbb{C}))\neq 0 \}+k.
 \end{eqnarray}
    \end{enumerate}
        \end{theorem}

    If  $X$  can  be  embedded  in  a  finite  dimensional  Euclidean  space,   
    then  by  \cite{high1,high2},  
    $X$   can  be $k$-regularly  embedded  in  certain  
    finite  dimensional  Euclidean  spaces    
    hence  for  any  $r\geq  0$,   $X$   can  be   $(k,r)$-regularly  embedded  in  
    certain  finite  dimensional  Euclidean  spaces.     
  In  this  case,  the  supremum   in  both  (\ref{eq-3.1x-introd-xx})  
    and  (\ref{eq-3.1y-introd-yy})  is  maximum.  
   For  example,   by  the Whitney  Embedding  Theorem,   
   any  smooth  manifold  $M$  can  be   
    $(k,r)$-regularly  embedded  in  certain  finite  dimensional  Euclidean  spaces.  
        Moreover,  since  any  graph $G$   can  be  embedded  in  $\mathbb{R}^3$,  
        $G$  can  be   
    $(k,r)$-regularly  embedded  in  certain  finite  dimensional  Euclidean  spaces. 
    Consequently,  for  both  manifolds  $M$  and  graphs  $G$,  
    the  supremum   in  both  (\ref{eq-3.1x-introd-xx})  
    and  (\ref{eq-3.1y-introd-yy})  is  maximum.

      Let  $(X,d)$  and  $(X',d')$  be  metric  spaces  such  that  $X\subseteq  X'$.  
        Then  $d$  is  inherited  from  $d'$   iff   
     $d(x,y)=d'(x,y)$  for  any  $x,y\in  X$.  
     In  particular,  if  we  let  $X'$  be  a  Riemannian  manifold  $M'$  and  let  
     $X$  be  a  Riemannian  submanifold  $M\subseteq  M'$,  
     then  we  call  $M$  a  {\it  strong  totally  geodesic  submanifold}  of  $M'$  if  
     $d_M(p,q)=d_{M'}(p,q)$  for any  $p,q\in  M$, 
     where  $d_M$   and  $d_{M'}$   are  the  geodesic  distances  of  $M$  and  $M'$  
     respectively.   
     The  next  theorem  is  a  consequence  of  Theorem~\ref{th-main-25-01-29}
      and  will  be  proved  in  Subsection~\ref{ss7.1}.

\begin{theorem} [Theorem~\ref{25-cor1zzz}] 
\label{25-cor1zzz-introd} 
Let  $M'$  be  a    Riemannian  manifold. 
Let $M$  run   over  all     strong  totally  geodesic   submanifolds  
 of   $M'$.   
\begin{enumerate}[(1)]
\item

If
 $f: M'\longrightarrow \mathbb{R}^N$  is  
 a $(k,r)$-regular map,   
 then 
  \begin{eqnarray*} 
 N\geq   \max\{t(k)\mid  \bar w_{t(k)}( \boldsymbol{\xi}(M,k,-+r;\mathbb{R}))\neq 0 \}+k;   
 \end{eqnarray*}
 
 \item
 If
 $f: M'\longrightarrow \mathbb{C}^N$  is   
 a  complex  $(k,r)$-regular map,   
 then 
  \begin{eqnarray*} 
 N\geq   \max\{t(k)\mid  \bar  c_{t(k)}( \boldsymbol{\xi}(M,k,-+r;\mathbb{C}))\neq 0 \}+k.   
 \end{eqnarray*}
 \end{enumerate} 
\end{theorem}

 Let  $G'$  be  a  graph  and  let  $G$  be  a  subgraph  of  $G'$.  
 Let  $|G|$  and   $|G'|$  be  
 the  geometric  realizations  of  $G$  and  $G'$  respectively,  i.e.    
  $|G|$  and   $|G'|$  are   the  $1$-dimensional  CW-complexes 
 consisting  of  the  vertices  and  the  edges  of  $G$  and  $G'$.     
 Then  $|G|$  is  a   subcomplex  of  $|G'|$  respectively.

 Let  $w':  E_{G'}\longrightarrow  (0,+\infty)$  be  a  weight  on  the  edges  of  $G'$.
 Then  $w'$  induces  a  metric  $d_{w'}$  on  $|G'|$.   
 The  restriction  of  $w'$  on  $E_G$  gives  a  weight $w$  on  the  edges  of  $G$.  
 This  induces a  metric  $d_w$  on  $|G|$.  
 Denote  $|G|_{w} =  (|G|, d_{w})$  and   $|G'|_{w'} =  (|G'|, d_{w'})$.  
 As  metric  spaces,  $|G|_w$  is  a  subspace  of  $|G'|_{w'}$.   
 The  next  theorem  is  a  consequence  of  Theorem~\ref{th-main-25-01-29}
      and  will  be  proved  in  Subsection~\ref{ss7.2}.  
 
 \begin{theorem}[Theorem~\ref{25-cor2zzz}]
 \label{co-25-01-graph-introd}
Let  $(G',w')$  be  a  graph  $G'$  with
  a   positive-valued   weight  $w'$  on  the  edges  of  $G'$.  
 \begin{enumerate}[(1)]
\item
If 
 $f: |G'|_{w'}\longrightarrow \mathbb{R}^N$ is  
 a $(k,r)$-regular map,    
 then 
  \begin{eqnarray*}\label{eq-3.1zzz-g-intro}
 N\geq   \max\{t(k)\mid  \bar w_{t(k)}( \boldsymbol{\xi}(|G|_w,k,-+r;\mathbb{R}))\neq 0 \}+k;  
 \end{eqnarray*}
 \item
 If  
 $f: |G'|_{w'}\longrightarrow \mathbb{C}^N$  is   
 a  complex  $(k,r)$-regular map,    
 then 
  \begin{eqnarray*}\label{eq-3.1hhh-g-intro}
 N\geq   \max\{t(k)\mid  \bar  c_{t(k)}( \boldsymbol{\xi}(|G|_w,k,-+r;\mathbb{C}))\neq 0 \}+k.   
 \end{eqnarray*}
 \end{enumerate} 
 Here  in  both  (1)  and  (2),      
    $G$  runs  over  all  the   subgraphs  
 of   $ G' $  satisfying  
  \begin{eqnarray*}
  d_w(x,y)=  d_{w'}(x,y) 
  \end{eqnarray*}
    for  any  $x,y\in  |G|_w$
 where    $w=w'|_G$.
\end{theorem}

    In  addition  of  Theorem~\ref{co-25-01-graph-introd},  
     if  the  weight  $w'$  is  integer-valued,  then  with  the  help  
     of  the  construction  of  the  independence  complex  by  configuration  spaces
      (cf.   Subsection~\ref{ss3.1}),  
     we  can  derive  the  followings  from        
     Theorem~\ref{co-25-01-graph-introd}~(1)  and  (2)  respectively   (cf.  Theorem~\ref{co-25-01-graph}).  

\begin{enumerate}[(1)']
\item
If
 $f: |G'|_{w'}\longrightarrow \mathbb{R}^N$ is   
 a $(k,r)$-regular map,   
 then 
  \begin{eqnarray*} 
 N\geq   
 \max\{t(k)\mid  \bar w_{t(k)}( \boldsymbol{\xi}(V_{{\rm  sd}(G;w)},k,-+r;\mathbb{R}))\neq 0 \}+k;  
 \end{eqnarray*}
 
 \item
If
 $f: |G'|_{w'}\longrightarrow \mathbb{C}^N$  is   
 a  complex  $(k,r)$-regular map,   
 then 
  \begin{eqnarray*} 
 N\geq  
  \max\{t(k)\mid  \bar  c_{t(k)}( \boldsymbol{\xi}(V_{{\rm  sd}(G;w)},k,-+r;\mathbb{C}))\neq 0 \}+k.  
 \end{eqnarray*}
 \end{enumerate} 
Here  $w$  is  the  restriction  of  $w'$  on  $G$  
and  ${\rm  sd}(G;w)$  is  the  subdivision  of  $G$
 with  respect  to  $w$  such  that  
 any  edge  $e\in  E_G$  is  divided  into  
 $w(e)$  edges.

    The    paper  is  organized  as  follows.  
    In  Section~\ref{25-s2},
     we    introduce  the  notions  of  persistent  
    bundles  over  persistent  CW-complexes  
    and  the  characteristic  classes  of  persistent  bundles.    
    In  Section~\ref{25-s3},  
    we    introduce   persistent  chain  complexes over  
    persistent  CW-complexes  and    persistent  manifolds
    as  well  as   persistent  double  chain  complexes    over  
    persistent  $\Delta$-CW-complexes
      and    persistent  $\Delta$-manifolds.    
    Both     Section~\ref{25-s2}  and Section~\ref{25-s3}
  are   
    applications   of    the     persistent homology  theory  
    and  give  preparations    for  later  sections.  
    In  Section~\ref{25-sect4},  
    we  construct  persistent    complexes  and  persistent  vector  bundles  
    for  the  configuration spaces  of  hard  spheres.  
    In  Section~\ref{s5},  
    we  use  the  persistent   bundles  over  configuration  spaces  
    to  give  obstructions  for  strong  totally  geodesic  embeddings.  
    In  Section~\ref{s7},  we   use  the  characteristic  classes  
    of  the  persistent  bundles  to  give  obstructions  for  the  
    $(k,r)$-regular  embeddings  
    and  prove  Theorem~\ref{th-main-25-01-29}  -  \ref{co-25-01-graph-introd}.  
     In  Section~\ref{s-777},  
     we  discuss  some   applications  of  the  regular embeddings  to  the  geometric  realization 
     problem   of  
the  independence  complexes.

 The  aim of  this  paper  is to give   some  prospective  tools   for  
 the  approximate    computations   of  the  $k$-regular  embedding  problems
     of  manifolds,  the  sphere-packing  problems  on  manifolds    
   and  the  geometric  realization  problems  of  the  independence  complexes  of  graphs,
   by using persistent homology.

\section{Persistent  fibre  bundles}\label{25-s2}

In  this  section,  we  introduce  the  definitions   of  persistent  covering  maps  (cf.  Definition~\ref{def-25-01-3}~(1)),  persistent   vector  bundles  (cf.  Definition~\ref{def-25-01-3}~(2))
and  persistent  characteristic  classes  (cf.  Definition~\ref{def-250120-1}  and  Definition~\ref{def-250120-2})  which  will  be  used  in  later  sections.   
All  the  maps  are  assumed  continuous.  

\begin{definition}
\label{def-0.za1}
A  {\it  persistent  CW-complex}  $\mathbf{A}$ 
 is  a  family  of  CW-complexes   $ A(t)$  for   $  t\in \mathbb{R} $
together  with  a  family  of    maps   
$\alpha(t_1,t_2):  A(t_1)\longrightarrow  A(t_2)$  for  $t_1\leq   t_2$  such  that  
\begin{enumerate}[(1)]
\item
 $\alpha(t ,t )$  is  the  identity  map  
    for  any  $t\in \mathbb{R}$;   
\item
   $\alpha(t_1,t_3)
 =\alpha(t_2,t_3)\circ  \alpha(t_1,t_2)$  for  any  $t_1\leq  t_2\leq  t_3$.  
 \end{enumerate}
\end{definition}

\begin{definition}\label{def-25-01-31-1}
The  {\it  persistent  cohomology  ring}  
$H^*(\mathbf{A})$  of  a   persistent  CW-complex  $\mathbf{A}$ 
 is  the  family  of  cohomology  rings  $H^*(A(t))$  for  $t\in \mathbb{R}$  
 together  with  the  family  of  induced  homomorphisms  
 $\alpha(t_1,t_2)^*:  H^*(A(t_2))\longrightarrow   H^*(A(t_1))$  for  $t_1\leq   t_2$.   
\end{definition}

The  persistent  cohomology  ring  
is  an  analog  of  the  persistent  homology  (cf.  \cite{cph1})  
with    extra   consideration  of    cup  products.  
For  the  persistent  cohomology  ring  $H^*(\mathbf{A})$  of  a  persistent  CW-complex
$\mathbf{A}$,     it  follows  from  Definition~\ref{def-0.za1}
  and  Definition~\ref{def-25-01-31-1}  that 
  \begin{enumerate}[(1)]
  \item
$ \alpha(t,t)^*$  is  the  identity  map  for  any  $t\in  \mathbb{R}$; 
 \item
     $\alpha(t_1,t_3)^*
 =  \alpha(t_1,t_2)^*\circ \alpha(t_2,t_3)^*$  for  any  $t_1\leq  t_2\leq  t_3$.
 \end{enumerate}

 Let  ${\bf  A} =\{A(t)\mid   t\in  \mathbb{R}\}$  
    and   ${\bf  B} =\{B(t)\mid   t\in  \mathbb{R}\}$  
    be  two  persistent  CW-complexes 
   with  maps  $\alpha(t_1,t_2):  A(t_1)\longrightarrow  A(t_2)$  and  
       $\beta(t_1,t_2):  B(t_1)\longrightarrow  B(t_2)$ respectively  for  any  $t_1\leq  t_2$.   
       
    \begin{definition}\label{def1ax}   
       A   {\it  persistent  map}  $f:  {\bf  A}\longrightarrow  {\bf   B}$  
       is  a  family   of  maps  $f(t):  A(t)\longrightarrow  B(t)$  for  any   $t\in \mathbb{R}$  
       such  that  the  diagram  
       \begin{eqnarray}\label{eq-diag-000}
       \xymatrix{
       A(t_1)\ar[d]_-{f(t_1)} \ar[rr] ^-{\alpha(t_1,t_2)}  &&  A(t_2)\ar[d]^-{f(t_2)}\\
       B(t_1)\ar[rr] ^-{\beta (t_1,t_2)} &&  B(t_2)
       }
       \end{eqnarray}
       commutes  for  any  $t_1\leq   t_2$.  
       \end{definition}

    Let   $f:  {\bf  A}\longrightarrow  {\bf   B}$   be   a     persistent  map.   
     Applying  the  cohomology  functor  to   
   $f $,  
we  have  an  induced  persistent  homomorphism  
\begin{eqnarray*}
f^*:   H^*(\mathbf{B})\longrightarrow  H^*(\mathbf{A}) 
\end{eqnarray*}
of  persistent  cohomology  rings  such  that  
  $
f(t)^*:   H^*(B(t))\longrightarrow  H^*(A(t)) 
  $  
is  a  homomorphism  of cohomology rings  for  any  $t\in \mathbb{R}$  and  
 the  diagram  
       \begin{eqnarray*}
       \xymatrix{
       H^*(A(t_1))  &  H^*(A(t_2))   \ar[l] _-{\alpha(t_1,t_2)^*}\\
       H^*(B(t_1)) \ar[u]^-{f(t_1)^*} & H^*(B(t_2)) \ar[l] _-{\beta(t_1,t_2)^* } \ar[u]_-{f(t_2)^*}
       }
       \end{eqnarray*}
       commutes  for  any  $t_1\leq   t_2$.

       Let  $\mathbb{F}=\mathbb{R}$  or    $\mathbb{C}$. 
       For  any  positive  integer  $k$,  let  $O(\mathbb{F}^k)$  be  the  
       orthogonal  group  of  the  Euclidean  space  $\mathbb{F}^k$.

        \begin{definition}\label{def-25-01-3}
       We call  $f:  {\bf  A}\longrightarrow  {\bf   B}$   a     {\it  persistent  fibre  bundle}  if    $f(t):  A(t)\longrightarrow  B(t)$ 
       is  a  fibre  bundle   for  any   $t\in \mathbb{R}$   and   the  diagram  (\ref{eq-diag-000})  
       is  a  pull-back  for     any  $t_1\leq   t_2$.    
       In  particular,     we call   $f:  {\bf  A}\longrightarrow  {\bf   B}$   
       \begin{enumerate}[(1)]
       \item
           a     {\it  persistent  covering  map}  if
           $f(t):  A(t)\longrightarrow  B(t)$ 
       is  a  covering  map   for  any   $t\in \mathbb{R}$;  
              \item
         a    {\it  persistent  vector  bundle}  if
        $f(t):  A(t)\longrightarrow  B(t)$ 
       is  a  vector  bundle   for  any   $t\in \mathbb{R}$;
       \item
         a 
       {\it  persistent  $O(\mathbb{F}^k)$-bundle}  if  $f(t):  A(t)\longrightarrow  B(t)$ 
       is  a  $O(\mathbb{F}^k)$-bundle   for  any   $t\in \mathbb{R}$.   
       \end{enumerate}   
       \end{definition}
       
       \begin{lemma}\label{le-000}
         Let   $f:  {\bf  A}\longrightarrow  {\bf   B}$   be   a     persistent  fibre  bundle.  
         Then   the  family  of  fibre  bundles   
         $F(t)\longrightarrow   A(t)\overset{f(t)}{\longrightarrow } B(t)$   
         has  a   uniform   fibre   $F(t)=F$  for  any  $t\in  \mathbb{R}$.   
       \end{lemma}
       
       \begin{proof}
        For     any  $t_1\leq   t_2$,  since   the  diagram  (\ref{eq-diag-000})  
       is  a  pull-back,  we  have  that     
       $F(t_1)\longrightarrow   A(t_1)\overset{f(t_1)}{\longrightarrow } B(t_1)$
       is  a  pull-back  fibre  bundle  
       of  
       $F(t_2)\longrightarrow   A(t_2)\overset{f(t_2)}{\longrightarrow } B(t_2)$.  
       Hence  
        $F(t_1)=F(t_2)$. 
       \end{proof}
       
         \begin{corollary}
        Let   $f:  {\bf  A}\longrightarrow  {\bf   B}$   be   a     persistent covering  map.  
       Then   the  family  of  covering  maps  $ f(t):   A(t) {\longrightarrow } B(t)$   
         has  a   constant  sheet    for  any  $t\in  \mathbb{R}$. 
       \end{corollary} 
       
       \begin{proof}
        For     any  $t_1\leq   t_2$,  
        by  Lemma~\ref{le-000},  
        the discrete  fibres  $F(t_1)$  and  $F(t_2)$  are  homeomorphic.  
        Hence  they  have  the  same  cardinality.  
              \end{proof}
       
       \begin{definition}\label{def-25-01-97}
       Two  persistent  fibre  bundles $f:  {\bf  A}\longrightarrow  {\bf   B}$   
       and   $f':  {\bf  A}'\longrightarrow  {\bf   B}'$   are  {\it  isomorphic}  if  
       \begin{enumerate}[(1)]
       \item
       the  fibre  bundles  
      $f(t):   A(t)\longrightarrow     B(t)$   
       and   $f'(t):    A' (t)\longrightarrow      B'(t)$   are  isomorphic 
        through  a  homeomorphism  $\varphi(t):  B(t)\longrightarrow  B'(t)$
        for  any  $t\in \mathbb{R}$;
       \item
       the  diagram  commutes  
       \begin{eqnarray*}
       \xymatrix{
     B(t_1)\ar[d]_-{\varphi(t_1)} \ar[rr] ^-{\beta(t_1,t_2)}  &&  B(t_2)\ar[d]^-{\varphi(t_2)}\\
       B'(t_1)\ar[rr] ^-{\beta' (t_1,t_2)} && B'(t_2)
       }
       \end{eqnarray*}
       for any  $t_1\leq  t_2$.  
       \end{enumerate}  
              \end{definition}
       
       By   Lemma~\ref{le-000},  we  can  denote  a  persistent fibre  bundle  
       $f:  {\bf  A}\longrightarrow  {\bf   B}$  as  
       \begin{eqnarray*}
    \boldsymbol{\xi}:    F\longrightarrow  {\mathbf{  A}}\overset{f}{\longrightarrow}  {\mathbf{B}}
       \end{eqnarray*}
       consisting  of  a  family  of  fibre  bundles  
       $
    \xi(t):    F\longrightarrow   A(t)\overset{f(t)}{\longrightarrow}   B(t)
   $
       for  $t\in  \mathbb{R}$  and    a  family  of   pull-backs  of  fibre  bundles 
       $
      \beta(t_1,t_2)^* :  \xi(t_2)\longrightarrow  \xi(t_1)
       $
       for  $t_1\leq  t_2$  such  that   (1)  and  (2)  
       after  Definition~\ref{def-25-01-31-1}  are  satisfied  by  substituting  $\alpha(-,-)^*$  with  
       $\beta(-,-)^*$.  
       
       
       \begin{definition}
       Let  $    \boldsymbol{\xi}:    \mathbb{F}^k\longrightarrow  {\mathbf{  E(  \boldsymbol{\xi})}}\overset{f}{\longrightarrow}  {\mathbf{B}}
$  and     
$ \boldsymbol{\eta}:    \mathbb{F}^l\longrightarrow  {\mathbf{  E(\boldsymbol{\eta})}}\overset{f}{\longrightarrow}  {\mathbf{B}}
$
 be   two  persistent  vector  bundles  over   the  same  persistent  CW-complex  $\mathbf{B}$.  
 \begin{enumerate}[(1)]
 \item
 Their  {\it  persistent  Whitney  sum}   $ \boldsymbol{\xi}\oplus  \boldsymbol{\eta}$  
is  the  persistent  vector  bundle  over  $\mathbf{B}$ 
  consisting  of  the  family  of  vector  bundles 
\begin{eqnarray*}
\xi(t)\oplus \eta(t):    \mathbb{F}^k\oplus   \mathbb{F}^l\longrightarrow  
E(\xi(t)\oplus \eta(t))\longrightarrow  B(t)
\end{eqnarray*}
for  $t\in \mathbb{R}$   together  with  the  family  of  pull-backs  of  vector  bundles 
\begin{eqnarray*}
\beta(t_1,t_2)^*:  \xi(t_2)\oplus\eta(t_2)\longrightarrow  \xi(t_1)\oplus\eta(t_1)
\end{eqnarray*}
such  that 
 $\beta(t_1,t_2)^*( \xi(t_2)\oplus\eta(t_2))=
 \beta(t_1,t_2)^*( \xi(t_2))\oplus\beta(t_1,t_2)^*(\eta(t_2))$ 
 for  $t_1\leq  t_2$;
\item  
Their  {\it  persistent  tensor  product}   $ \boldsymbol{\xi}\otimes  \boldsymbol{\eta}$  
 is  the  persistent  vector  bundle  over  $\mathbf{B}$ 
       consisting  of  the  family  of  vector  bundles 
\begin{eqnarray*}
\xi(t)\otimes \eta(t):    \mathbb{F}^k\otimes   \mathbb{F}^l\longrightarrow  
E(\xi(t)\otimes \eta(t))\longrightarrow  B(t)
\end{eqnarray*}
for  $t\in \mathbb{R}$   together  with  the  family  of  pull-backs  of  vector  bundles 
\begin{eqnarray*}
\beta(t_1,t_2)^*:  \xi(t_2)\otimes\eta(t_2)\longrightarrow  \xi(t_1)\otimes\eta(t_1)
\end{eqnarray*}
such  that 
 $\beta(t_1,t_2)^*( \xi(t_2)\otimes\eta(t_2))=
 \beta(t_1,t_2)^*( \xi(t_2))\otimes\beta(t_1,t_2)^*(\eta(t_2))$
  for  $t_1\leq  t_2$.  
\end{enumerate}
\end{definition}

\begin{corollary}\label{cor-2.2}
 Let  $    \boldsymbol{\xi}:    \mathbb{F}^k\longrightarrow  {\mathbf{  E(  \boldsymbol{\xi})}}\overset{f}{\longrightarrow}  {\mathbf{B}}
$  be  a  persistent  $O(\mathbb{F}^k)$-bundle  and     
$ \boldsymbol{\eta}:    \mathbb{F}^l\longrightarrow  {\mathbf{  E(\boldsymbol{\eta})}}\overset{f}{\longrightarrow}  {\mathbf{B}}
$   be  a  persistent  $O(\mathbb{F}^l)$-bundle.  
Then  their  persistent  Whitney  sum  $\boldsymbol{\xi}\oplus\boldsymbol{\eta}$  
is  a  persistent  $O(\mathbb{F}^{k+l})$-bundle and  
their  persistent  tensor  product  $\boldsymbol{\xi}\otimes\boldsymbol{\eta}$  
is  a  persistent  $O(\mathbb{F}^{kl})$-bundle.  
\end{corollary}

\begin{proof}
For  each  $t\in  \mathbb{R}$,  
the  Whitney  sum  $\xi(t)\oplus  \eta(t)$  is  a  $O(\mathbb{F}^{k+l})$-bundle.  
Thus  the  persistent  vector  bundle  $\boldsymbol{\xi}\oplus\boldsymbol{\eta}$  
is  a  persistent  $O(\mathbb{F}^{k+l})$-bundle. 
For  each  $t\in  \mathbb{R}$,  
the  tensor  product  $\xi(t)\otimes  \eta(t)$  is  a  $O(\mathbb{F}^{kl})$-bundle.  
Thus  the  persistent  vector  bundle  $\boldsymbol{\xi}\otimes\boldsymbol{\eta}$  
is  a  persistent  $O(\mathbb{F}^{kl})$-bundle. 
\end{proof}

       \begin{lemma}\label{le-0.ba1}
       Let  $
    \boldsymbol{\xi}:    \mathbb{F}^k\longrightarrow  {\mathbf{  A}}\overset{f}{\longrightarrow}  {\mathbf{B}}
     $
       be  a  persistent  $O(\mathbb{F}^k)$-bundle.  
       \begin{enumerate}[(1)]
       \item
       Suppose  $\mathbb{F}=\mathbb{R}$.  
       Take  the  persistent  cohomology ring  $H^*(\mathbf{B};\mathbb{Z}_2)$  with   coefficients  
       in  $\mathbb{Z}_2$.   Then  
       \begin{eqnarray}\label{eq-0.cc1}
       \beta(t_1,t_2)^* (w(\xi(t_2))) =  w (\xi(t_1))
       \end{eqnarray}
       for  any  $t_1\leq  t_2$,  where  $w(-)$  is  the  total  Stiefel-Whitney  class;   
       \item
              Suppose  $\mathbb{F}=\mathbb{C}$.  
Take  the  persistent  cohomology ring  $H^*(\mathbf{B};\mathbb{Z} )$  with   coefficients  
       in  $\mathbb{Z} $.    Then  
       \begin{eqnarray}\label{eq-0.cc2}
       \beta(t_1,t_2)^* (c(\xi(t_2))) =  c (\xi(t_1))
       \end{eqnarray}
       for  any  $t_1\leq  t_2$,  where  $c(-)$  is  the  total  Chern  class.   
       \end{enumerate}  
       \end{lemma}
       
       \begin{proof}
       Let  $t_1\leq  t_2$.  
       Then   $\xi(t_1)=  \beta(t_1,t_2)^* ( \xi(t_2) )$  is  the  pull-back  of   $O(\mathbb{F}^k)$-bundles.     
       (1)  Suppose   $\mathbb{F}=\mathbb{R}$.  
        By  the  functoriality  of  the  Stiefel-Whitney  class,  we  obtain  (\ref{eq-0.cc1}).   
        (2)   Suppose   $\mathbb{F}=\mathbb{C}$.  
        By  the  functoriality  of  the  Chern  class,  we  obtain  (\ref{eq-0.cc2}).
       \end{proof}
       
       The  next    definition  follows   from   Lemma~\ref{le-0.ba1}~(1).  
       
       \begin{definition}\label{def-250120-1}
        Let  $
    \boldsymbol{\xi}:    \mathbb{R}^k\longrightarrow  {\mathbf{  A}}\overset{f}{\longrightarrow}  {\mathbf{B}}
     $
       be  a  persistent  $O(\mathbb{R}^k)$-bundle.  
       We  define  the  {\it  persistent  Stiefel-Whitney  class}  
       \begin{eqnarray*}
       w(  \boldsymbol{\xi}) = 1+ w_1   (  \boldsymbol{\xi})+ \cdots +   w_k(  \boldsymbol{\xi})
       \end{eqnarray*}
         to  be the  family  of  Stiefel  Whitney  classes 
       \begin{eqnarray*}
       w(\xi(t)) =1+ w_1   (  \xi(t))+ \cdots +   w_k(  \xi(t))
              \end{eqnarray*}
         for  $t\in  \mathbb{R}$  such  that  (\ref{eq-0.cc1})  is  satisfied  for  $t_1\leq  t_2$.   
      \end{definition}

      \begin{corollary}
        Let  $
    \boldsymbol{\xi}:    \mathbb{R}^k\longrightarrow  {\mathbf{  A}}\overset{f}{\longrightarrow}  {\mathbf{B}}
     $
      and  
      $
    \boldsymbol{\xi}':    \mathbb{R}^k\longrightarrow  {\mathbf{  A}'}\overset{f}{\longrightarrow}  {\mathbf{B}'}
     $
       be  two  persistent  $O(\mathbb{R}^k)$-bundles.  
      If  $ \boldsymbol{\xi}$  and  $ \boldsymbol{\xi}'$  are  isomorphic,  
      then  $w(  \boldsymbol{\xi}) =w(  \boldsymbol{\xi}')$.   
      \end{corollary}
      
      \begin{proof}
      The    corollary  follows  from  Definition~\ref{def-25-01-97}
      and  Definition~\ref{def-250120-1}.  
      \end{proof}

      \begin{corollary}\label{cor-2.5a}
       Let  $    \boldsymbol{\xi}:    \mathbb{R}^k\longrightarrow  {\mathbf{  E(  \boldsymbol{\xi})}}\overset{f}{\longrightarrow}  {\mathbf{B}}
$  be  a  persistent  $O(\mathbb{R}^k)$-bundle  and     
$ \boldsymbol{\eta}:    \mathbb{R}^l\longrightarrow  {\mathbf{  E(\boldsymbol{\eta})}}\overset{f}{\longrightarrow}  {\mathbf{B}}
$   be  a  persistent  $O(\mathbb{R}^l)$-bundle.  
Then  
\begin{eqnarray}
\label{eq-sw.1}
 w(  \boldsymbol{\xi}\oplus \boldsymbol{\eta}) = w(\boldsymbol{\xi})  w(\boldsymbol{\eta})   
 \label{eq-sw.2}
\end{eqnarray}
where  the  righthand  side  of  (\ref{eq-sw.1})  is  the  cup-product  in  the  mod  $2$  persistent  
 cohomology  of  $ {\mathbf{B}}$. 
      \end{corollary}
      
      \begin{proof}
      Take  the  persistent  cohomology  ring  $H^*(\mathbf{B};\mathbb{Z}_2)$  
      of  $\mathbf{B}$  with  coefficients in  $\mathbb{Z}_2$.  
      We  have  $ w(   \xi(t)\oplus  \eta(t)) = w( \xi (t))  w( \eta(t) )$  for  any  $t\in  \mathbb{R}$. 
      With  the  help  of  Corollary~\ref{cor-2.2},  we  obtain  (\ref{eq-sw.1}).   
      \end{proof}
      
             The  next    definition  follows   from   Lemma~\ref{le-0.ba1}~(2).

      \begin{definition}\label{def-250120-2}
       Let  $
    \boldsymbol{\xi}:    \mathbb{C}^k\longrightarrow  {\mathbf{  A}}\overset{f}{\longrightarrow}  {\mathbf{B}}
     $
       be  a  persistent  $O(\mathbb{C}^k)$-bundle.  
       We  define  the  {\it  persistent  Chern  class}  
       \begin{eqnarray*}
       c(  \boldsymbol{\xi})  =1 +  c_1(  \boldsymbol{\xi})+ \cdots  +c_k(  \boldsymbol{\xi})
       \end{eqnarray*}
         to  be the  family  of  Chern  classes 
       \begin{eqnarray*}
       c(\xi(t))=1+ c_1 (\xi(t))+\cdots  +c_k(\xi(t))
       \end{eqnarray*}
         for  $t\in  \mathbb{R}$  such  that  (\ref{eq-0.cc2})  is  satisfied  for  $t_1\leq  t_2$.   
       \end{definition}

      \begin{corollary}
         Let  $
    \boldsymbol{\xi}:    \mathbb{C}^k\longrightarrow  {\mathbf{  A}}\overset{f}{\longrightarrow}  {\mathbf{B}}
     $
      and  
      $
    \boldsymbol{\xi}':    \mathbb{C}^k\longrightarrow  {\mathbf{  A}'}\overset{f}{\longrightarrow}  {\mathbf{B}'}
     $
       be  two  persistent  $O(\mathbb{C}^k)$-bundles.  
      If  $ \boldsymbol{\xi}$  and  $ \boldsymbol{\xi}'$  are  isomorphic,  
      then  $c(  \boldsymbol{\xi}) =c(  \boldsymbol{\xi}')$.   
      \end{corollary}
      
      \begin{proof}
         The   corollary  follows  from  Definition~\ref{def-25-01-97}
      and  Definition~\ref{def-250120-2}.  
      \end{proof}

\begin{corollary}\label{cor-2.7a}
 Let  $    \boldsymbol{\xi}:    \mathbb{C}^k\longrightarrow  {\mathbf{  E(  \boldsymbol{\xi})}}\overset{f}{\longrightarrow}  {\mathbf{B}}
$  be  a  persistent  $O(\mathbb{C}^k)$-bundle  and     
$ \boldsymbol{\eta}:    \mathbb{C}^l\longrightarrow  {\mathbf{  E(\boldsymbol{\eta})}}\overset{f}{\longrightarrow}  {\mathbf{B}}
$   be  a  persistent  $O(\mathbb{C}^l)$-bundle.  
Then  
\begin{eqnarray}
\label{eq-c.1}
 c(  \boldsymbol{\xi}\oplus \boldsymbol{\eta}) = c(\boldsymbol{\xi})    c(\boldsymbol{\eta}) 
\end{eqnarray}
where  the  righthand  side  of  (\ref{eq-c.1})  is  the  cup-product  in  the  integral  persistent  
 cohomology  of  $ {\mathbf{B}}$. 
\end{corollary}

   \begin{proof}
      Take  the  persistent  cohomology  ring  $H^*(\mathbf{B};\mathbb{Z})$  
      of  $\mathbf{B}$  with  coefficients in  $\mathbb{Z}$.  
      We  have  $ c(   \xi(t)\oplus  \eta(t)) = c( \xi(t) )  c( \eta(t) )$  for  any  $t\in  \mathbb{R}$. 
      With  the  help  of  Corollary~\ref{cor-2.2},  we  obtain  (\ref{eq-c.1}).   
      \end{proof}

\section{Persistent   complexes  and  persistent  manifolds} \label{25-s3}

In  this  section,  we  construct  associated  persistent  chain  complexes
(cf.  Definition~\ref{def-250112-7})  for  
persistent  $\Delta$-sets   (cf.  Definition~\ref{def-7.1aa1})  and  construct  associated  persistent  double  complexes  (cf.  Definition~\ref{def-25-02-08})   for  
persistent  bi-$\Delta$-sets  (cf.  Definition~\ref{def-7.187}).  
Then  we    
construct  persistent  double  complexes  for   persistent  $\Delta$-CW-complexes 
(cf.  Definition~\ref{def-7.3})   
and    persistent   $\Delta$-manifolds  (cf.  Definition~\ref{def-001a1}).

\subsection{Persistent  $\Delta$-sets  and  persistent  chain  complexes}

\begin{definition}(cf.  \cite[Chap.~I]{s-s-book})
\label{def-7.1-hrshdjtx}
A  {\it  $\Delta$-set}    is  a  sequence  of  sets  $S_\bullet=(S_n)_{n\in\mathbb{N}}$  with  face  maps 
$\partial^i_n:  S_n\longrightarrow  S_{n-1}$,  $i=0,1,\ldots,n$, 
  satisfying   the $\Delta$-identity
\begin{eqnarray}\label{eq-250112-5} 
\partial_{n-1}^i \partial_n^j= \partial_{n-1}^{j-1} \partial  _n^i 
\end{eqnarray}
  for  any   $0\leq  i<j\leq  n$. 
  \end{definition}   
  
  \begin{definition}\label{def-7.1aa1}
A  {\it  persistent  $\Delta$-set}   $\mathbf{S}_\bullet=\{S_\bullet(t)  \mid  t\in \mathbb{R}\} $  is  a  family   of  $\Delta$-sets  $S_\bullet(t)$  whose  face maps  are $\partial^i_n(t)$
  for  $t\in  \mathbb{R}$   together 
 with  a  family  of  maps  
 \begin{eqnarray}\label{eq-250112-7}
 \alpha_\bullet(t_1,t_2):  S_\bullet(t_1)\longrightarrow  S_\bullet(t_2)
 \end{eqnarray}
    for  $t_1\leq  t_2$  
 such  that  Definition~\ref{def-0.za1}~(1)  and  (2)  are  satisfied   and    
   $ \alpha_{n-1}(t_1,t_2)\circ \partial^i_n(t_1) = \partial^i_n(t_2) \circ \alpha_n(t_1,t_2)$,  i.e. the  diagram  commutes
    \begin{eqnarray}\label{eq-diag-3.a1}
    \xymatrix{
     S_n(t_1)\ar[rr]^-{\alpha_n(t_1,t_2)}\ar[d]_-{\partial^i_n(t_1)}  && 
      S_n(t_2)\ar[d]^-{\partial^i_n(t_2)} \\
     S_{n-1}(t_1)\ar[rr]^-{\alpha_{n-1}(t_1,t_2)}   &&  S_{n-1}(t_2) 
    }
    \end{eqnarray}
    for  any  $n\in  \mathbb{N}$,  any  $1\leq  i\leq  n$  and  any  $t_1\leq  t_2$.
  \end{definition}

A  {\it  chain  complex}   $C_\bullet=(C_n)_{n\in  \mathbb{Z}}$  
consists  of  a  sequence  of  abelian  groups  
$C_n$ 
together  with  a  sequence  of  homomorphisms  
 $\partial_n:  C_n\longrightarrow  C_{n-1}$  such  that  $\partial_{n} \partial_{n+1}=0$     
  for  any  $n\in  \mathbb{Z}$    (cf.  \cite[p. 106]{at}).   
  A  {\it     chain  map}  $\varphi:   C_\bullet\longrightarrow  C'_\bullet$  from a  chain  complex  $C_\bullet$  
  to a  chain  complex  $C'_\bullet$  is  a  sequence  of  homomorphisms  $\varphi_n:  C_n\longrightarrow C'_n$
    such  that  
   $\partial_n  \circ  \varphi_n=\varphi_{n-1}\circ \partial_n$  for  any $n\in  \mathbb{Z}$   (cf.  \cite[p. 111]{at}).   
 The  following  definition  is  a  continuous  version  of   \cite[Definition~3.1]{cph1}.

\begin{definition} 
\label{def-250112-7}
  A  {\it   persistent   chain  complex}   $\mathbf{C}_\bullet= \{C_\bullet(t)\mid  t\in \mathbb{R}\}$  
is  a  family  of  chain  complexes  $C_\bullet(t)$  for  $t\in \mathbb{R}$  together  
with  a  family  of  chain  maps  $\varphi(t_1,t_2):   C_\bullet(t_1)\longrightarrow  C_\bullet(t_2)$  for  $t_1\leq  t_2$  
such  that  
\begin{enumerate}[(1)]
\item
 $\varphi(t,t)$  is  the identity  map  for  any  $t\in  \mathbb{R}$;
 \item  
 $\varphi(t_1,t_3)= \varphi(t_2,t_3) \circ  \varphi(t_1,t_2)$  for any  $t_1\leq  t_2\leq  t_3$.   
 \end{enumerate}
\end{definition}

 Let     $\mathbf{S}_\bullet=\{S_\bullet(t)  \mid  t\in \mathbb{R}\} $  be  a    persistent  $\Delta$-set. 
     Let  $R$  be  a  commutative ring with  unit.  
For  any   $n\in \mathbb{N}$  and  any  $t\in  \mathbb{R}$,  
 let   
 $C(S_n(t)) $  be  the  free  $R$-module  generated  by  
 all  the  elements  in  $S_n(t)$.  
 Let  
 \begin{eqnarray*}
 \partial_n^i(t)_\#:  C(S_n(t)) \longrightarrow  C(S_{n-1}(t))  
 \end{eqnarray*}  
 be  the  induced  homomorphism  of  $R$-modules  for  any  $0\leq  i\leq  n$.  Let  
 \begin{eqnarray}\label{eq-250112-2}
 \partial_n(t)=\sum_{i=0}^n (-1)^i  \partial_n^i(t)_\#.  
 \end{eqnarray}
 Let  
 \begin{eqnarray}\label{eq-250112-1}
 \alpha_\bullet(t_1,t_2)_\#:  C(S_\bullet(t_1)) \longrightarrow  C(S_\bullet(t_2))
 \end{eqnarray}
 be  the   homomorphism  of  $R$-modules  induced  from
  (\ref{eq-250112-7}).

 \begin{lemma}\label{le-250112-1}
 Let     $\mathbf{S}_\bullet=\{S_\bullet(t)  \mid  t\in \mathbb{R}\} $  be  a    persistent  $\Delta$-set.  
 Then  we  have  an  associated  persistent  chain  complex  
 \begin{eqnarray}\label{eq-003.1vbxx22}
C(\mathbf{S}_\bullet) =\{C(S_\bullet(t) ) \mid  t\in \mathbb{R}\} 
 \end{eqnarray}
 with  the  persistent   boundary  map  $\boldsymbol{\partial}_\bullet = \{ \partial_\bullet(t)\mid  t\in \mathbb{R}\}$.   
 \end{lemma}
 \begin{proof}
It  follows from  (\ref{eq-250112-5})  and  (\ref{eq-250112-2})   that    
$
\partial_{n-1}(t)\circ    \partial_n(t)=0
$
  for  any  $n\in \mathbb{N}$
and  any  $t\in \mathbb{R}$.  
Hence  $C(S_\bullet(t) )$  is  a  chain  complex  for any  $t\in \mathbb{R}$.   
By  (\ref{eq-diag-3.a1}),  
\begin{eqnarray}\label{eq-250112-3}
\alpha_{n-1}(t_1,t_2)_\#\circ  \partial_n^i(t_1)_\# =    \partial_n^i(t_2)_\#  \circ  \alpha_n(t_1,t_2)_\#
\end{eqnarray}
      for   any  $n\in \mathbb{N}$, any  $0\leq  i\leq  n$    and  any  $t_1\leq  t_2$.  
      It  follows  from  (\ref{eq-250112-2})   and   (\ref{eq-250112-3})   that  
      \begin{eqnarray*}
\alpha_{n-1}(t_1,t_2)_\#\circ  \partial_n(t_1)  =    \partial_n(t_2) \circ  \alpha_n(t_1,t_2)_\#
\end{eqnarray*}
       for   any  $n\in \mathbb{N}$   and  any  $t_1\leq  t_2$.  
       Hence   (\ref{eq-250112-1})  
       is  a  chain  map  for  any  $t_1\leq  t_2$.   
       It  follows from  Definition~\ref{def-0.za1}~(1)   and  (2)  respectively  that  
       Definition~\ref{def-250112-7}~(1)  and  (2)  are  satisfied  for  
       (\ref{eq-250112-1}).  
   Therefore,         
     (\ref{eq-003.1vbxx22})  is  a  persistent  chain  complex.  
        \end{proof}

\begin{definition}(cf.  \cite[Chap.~IV]{s-s-book})
\label{def-7.187}
A  {\it  bi-$\Delta$-set}    is  a  bigraded  sequence  of  sets  $S_{\bullet,\bullet}=(S_{p,q})_{p,q\in\mathbb{N}}$  with  face  maps 
$\partial^i_{p,q}:  S_{p,q}\longrightarrow  S_{p-1,q}$,  $i=0,1,\ldots,p$,  and  
$d^j_{p,q}:  S_{p,q}\longrightarrow  S_{p,q-1}$,    $j=0,1,\ldots,q$, 
 satisfying   the $\Delta$-identities  
\begin{eqnarray}
\partial_{p-1,q}^i \partial_{p,q}^j&=& \partial _{p-1,q}^{j-1} \partial  _{p,q}^i, 
\label{eq-250112-01}\\
d_{p,q-1}^k d_{p,q}^l &=&  d_{p,q-1}^{l-1}  d_{p,q}^k
\label{eq-250112-02}
\end{eqnarray}
  for  any   $0\leq  i<j\leq  p$ and  any  $0\leq  k<l\leq  q$ 
  such  that  the  diagram  commutes 
  \begin{eqnarray}\label{eq-250112-03}
  \xymatrix{
  S_{p,q} \ar[r]  ^-{\partial_{p,q}^i} \ar[d]_-{d_{p,q}^k}   &S_{p-1,q} \ar[d]^-{d_{p-1,q}^k}\\
  S_{p,q-1} \ar[r] ^-{\partial_{p,q-1}^i}  & S_{p-1,q-1}
  }
  \end{eqnarray}
  for  any  $p,q\in \mathbb{N}$,  any  $0\leq  i\leq  p$  and  any  $0\leq  k\leq  q$.  
  \end{definition}   

 \begin{definition}
 A  {\it  persistent    bi-$\Delta$-set}    
 $\mathbf{S}_{\bullet,\bullet} =\{S_{\bullet,\bullet}(t)\mid  t\in\mathbb{R}\}$  
 is  a  family  of  bi-$\Delta$-sets  $S_{\bullet,\bullet}(t)$  whose  face  maps  
 are  $\partial^i_{p,q}(t)$  and  $d^k_{p,q}(t)$  for  $t\in\mathbb{R}$  
 together  with  a  family  of  maps  
 \begin{eqnarray}\label{eq-250112-9}
 \alpha(t_1,t_2):  S_{\bullet,\bullet}(t_1)\longrightarrow  S_{\bullet,\bullet}(t_2)
 \end{eqnarray}  
 such  that  Definition~\ref{def-0.za1}~(1) and  (2) are  satisfied  and  
 \begin{eqnarray}
 \alpha_{p-1,q}(t_1,t_2)\circ \partial^i_{p,q}(t_1) &=& \partial^i_{p,q}(t_2) \circ \alpha_{p,q}(t_1,t_2),
 \label{eq-250112-27}\\ 
  \alpha_{p,q-1}(t_1,t_2)\circ d^k_{p,q}(t_1) &=& d^k_{p,q}(t_2) \circ \alpha_{p,q}(t_1,t_2)
     \label{eq-250112-28}
 \end{eqnarray}
 for  any  $p,q\in \mathbb{N}$,  any  $0\leq  i\leq  p$,  any  $0\leq  k\leq  q$ 
  and  any  $t_1\leq  t_2$.  
 \end{definition}
 
 The   (persistent)  double  complex  is   a   generalization  of  the  (persistent)  chain  complex.  
A  {\it  double   complex}  $C_{\bullet,\bullet}=(C_{p,q})_{p,q\in  \mathbb{Z}}$ 
 consists  of  a  bigraded  family   of  abelian  groups  $C_{p,q}$  together  with 
 two   families  of  homomorphisms  
 $\partial:  C_{p,q}\longrightarrow  C_{p,q-1}$  and  $d:  C_{p,q}\longrightarrow  C_{p-1,q}$  
 for  any  $p,q\in  \mathbb{Z}$  such  that  
 $\partial^2=d^2=0$  and  $\partial  \circ d=d\circ \partial$.  
   A  {\it     morphism}  $\varphi:   C_{\bullet,\bullet}\longrightarrow  C'_{\bullet,\bullet}$ 
    from a  double     complex  $C_{\bullet,\bullet}$  
  to a  double  complex  $C'_{\bullet,\bullet}$  is  a  bigraded  family 
    of  homomorphisms  $\varphi_{p,q}:  C_{p,q}\longrightarrow C'_{p,q}$
    such  that  
   $\partial  \circ  \varphi_{p,q}=\varphi_{p-1,q}\circ \partial $   and  
     $d \circ  \varphi_{p,q}=\varphi_{p,q-1}\circ d$  
     for  any $p,q\in  \mathbb{Z}$.

\begin{definition}\label{def-25-02-08}
  A  {\it   persistent    double  complex}   $\mathbf{C}_{\bullet,\bullet}= \{C_{\bullet,\bullet}(t)\mid  t\in \mathbb{R}\}$  
is  a  family  of  double  complexes  $C_{\bullet,\bullet}(t)$  for  $t\in \mathbb{R}$  together  
with  a  family  of   morphisms   $\varphi(t_1,t_2):   C_{\bullet,\bullet}(t_1)\longrightarrow  C_{\bullet,\bullet}(t_2)$  for  $t_1\leq  t_2$  
such  that  
 Definition~\ref{def-250112-7}~(1)  and  (2)  are  satisfied.     
\end{definition}

Let     $\mathbf{S}_{\bullet,\bullet}=\{S_{\bullet,\bullet}(t)  \mid  t\in \mathbb{R}\} $  be  a    persistent  
bi-$\Delta$-set. 
For  any   $p,q\in \mathbb{N}$  and  any  $t\in  \mathbb{R}$,  
 let   
 $C(S_{p,q}(t)) $  be  the  free  $R$-module  generated  by  
 all  the  elements  in  $S_{p,q}(t)$.  
 Let  
 \begin{eqnarray*}
 \partial_{p,q}^i(t)_\#:  &&  C(S_{p,q}(t)) \longrightarrow  C(S_{p-1,q}(t)),\\ 
 d_{p,q}^k(t)_\#:  &&  C(S_{p,q}(t)) \longrightarrow  C(S_{p,q-1}(t))
 \end{eqnarray*}  
 be  the  induced  homomorphisms  of  $R$-modules  for  any  $0\leq  i\leq  p$ 
 and  any  $0\leq  k\leq  q$.  Let  
 \begin{eqnarray*}
 \partial_{p,q}(t)&=&\sum_{i=0}^p (-1)^i  \partial_{p,q}^i(t)_\#,\\
  d_{p,q}(t)&=& \sum_{k=0}^q (-1)^k d_{p,q}^k(t)_\#.  
 \end{eqnarray*}
 Let  
 \begin{eqnarray}\label{eq-250112-22}
 \alpha(t_1,t_2)_\#:  C(S_{\bullet,\bullet}(t_1)) \longrightarrow  C(S_{\bullet,\bullet}(t_2))
 \end{eqnarray}
 be  the   homomorphism  of  $R$-modules  induced  by  (\ref{eq-250112-9}).  
  The  next  lemma  is  a  generalization of  Lemma~\ref{le-250112-1}.

 \begin{lemma}\label{le-250112-7}
 Let     $\mathbf{S}_{\bullet,\bullet}=\{S_{\bullet,\bullet}(t)  \mid  t\in \mathbb{R}\} $  be  a    persistent 
  bi-$\Delta$-set.  
 Then  we  have  an  associated  persistent  double   complex  
 \begin{eqnarray}\label{eq-003.1vb}
C(\mathbf{S}_{\bullet,\bullet}) =\{C (S_{\bullet,\bullet}(t) ) \mid  t\in \mathbb{R}\} 
 \end{eqnarray}
 with  the  persistent   boundary  maps  $\boldsymbol{\partial}_{p,q} = \{  \partial_{p,q}(t)\mid  t\in \mathbb{R}\}$  and  $\mathbf{d}_{p,q} = \{  d_{p,q}(t)\mid  t\in \mathbb{R}\}$.   
 \end{lemma}
 \begin{proof}
 It  follows  from  (\ref{eq-250112-01}),   (\ref{eq-250112-02})  
 and   (\ref{eq-250112-03})
  respectively  that  
   \begin{eqnarray*}
 &  \partial_{p,q}(t)^2=0, ~~~~~~  d_{p,q}(t)^2=0, \\
 &  \partial_{p,q-1}(t)  \circ d_{p,q}(t)=d_{p-1,q}(t)\circ \partial_{p,q}(t)  
   \end{eqnarray*}
 for  any  $p,q\in \mathbb{N}$  and   any  $t\in \mathbb{R}$.    
It  follows  from  (\ref{eq-250112-27})  and  (\ref{eq-250112-28})  
respectively  that  
\begin{eqnarray*}
 \partial_{p,q}(t_2)  \circ  \alpha_{p,q}(t_1,t_2)_\# &=& \alpha_{p-1,q}(t_1,t_2)_\#\circ \partial_{p,q}(t_1),      \\  
 d_{p,q}(t_2)  \circ  \alpha_{p,q}(t_1,t_2)_\# &=& \alpha_{p,q-1}(t_1,t_2)_\#\circ  d_{p,q}(t_1)   
\end{eqnarray*}
for any  $p,q\in \mathbb{N}$ and  any  $t_1\leq  t_2$.  
Hence  (\ref{eq-250112-22})  is  a  morphism  of  double  complexes  for  any  $t_1\leq  t_2$.  
   It  follows from  Definition~\ref{def-0.za1}~(1)   and  (2)  respectively  that  
       Definition~\ref{def-250112-7}~(1)  and  (2)  are  satisfied  for  
       (\ref{eq-250112-22}).  
   Therefore,         
  (\ref{eq-003.1vb})  is  a  persistent  double  complex.  
  \end{proof}

\subsection{Persistent  $\Delta$-CW-complexes}
 
\begin{definition}\label{def-7.2}
 \footnote[2]{The     notion  of   $\Delta$-CW-complexes   here   is   different  from  
 the  notion  of  
 $\Delta$-complexes  in  \cite[p.  103]{at}.   }
A  {\it  $\Delta$-CW-complex}     is  a $\Delta$-set  $A_\bullet= (A_n)_{n\in\mathbb{N}}$
 where  all the  sets  $A_n$  are  CW-complexes  and  all the  face  maps 
 $\partial_n^i:  A_n\longrightarrow  A_{n-1}$   are  cellular.
   \end{definition}   
 
\begin{definition}\label{def-7.3}
 A   {\it   persistent  $\Delta$-CW-complex}    is  a   family  of   $\Delta$-CW-complexes
  \begin{eqnarray*}
  \mathbf{A}_\bullet= \{  A_\bullet(t)\mid   t\in  \mathbb{R}\},  
  \end{eqnarray*}  
 where  $A_\bullet(t)$  is  a  $\Delta$-CW-complex  with  face  maps  
   $\partial^i_n(t)$  for     $t\in  \mathbb{R}$,  together  with  a  family  of   cellular    
   maps  
   \begin{eqnarray}\label{eq-3.9az1}
    \alpha_n(t_1,t_2):   A_n(t_1)\longrightarrow  A_n(t_2)
    \end{eqnarray} 
    for    $n\in  \mathbb{N}$    and    $t_1\leq  t_2$  such  that  
    the  identities  in  
      Definition~\ref{def-0.za1}~(1), (2)  and 
       the  commutative  diagram  (\ref{eq-diag-3.a1})   are  satisfied
          for  any    $0\leq  i\leq  n$.    
    \end{definition}

    Let  $\mathbf{A}_\bullet=  \{A_\bullet(t)\mid  t\in \mathbb{R}\}$  
    be  a  persistent  $\Delta$-CW-complex,  where  $A_\bullet(t)= (A_n(t))_{n\in \mathbb{N}}$  
    is  a  $\Delta$-CW-complex    with  face  maps  
    \begin{eqnarray*}
    \partial_n^i(t):  A_n(t)\longrightarrow  A_{n-1}(t)    
    \end{eqnarray*}
    for    $t\in    \mathbb{R}$.    
     Let  $(C_\bullet(A_n(t)), d(t))$  be  the   cellular  chain  complex  
    with 
    $C_m(A_n(t))$  the  free   $R$-module  spanned  by  the  $m$-cells  of  $A_n(t)$ 
    and  the  
     boundary  maps 
    \begin{eqnarray*}
    d_m(t):  C_m(A_n(t))\longrightarrow  C_{m-1}(A_n(t))
    \end{eqnarray*}
   for    $m\in \mathbb{N}$  (cf.  \cite[p. 139  and  p.  153]{at}).  
    The  next  proposition  is  a  topological  version  of     
       Lemma~\ref{le-250112-7}.  
    
    \begin{proposition}\label{pr-3.2.a1}
    Let  $\mathbf{A}_\bullet=  \{A_\bullet(t)\mid  t\in \mathbb{R}\}$  
    be  a  persistent  $\Delta$-CW-complex.  
   Then   we  have   
      a  persistent  double  complex  
 \begin{eqnarray}\label{eq-3.lllq2}
(C_\bullet(\mathbf{A}_\bullet), \boldsymbol{\partial},  \mathbf{d}) =\{(C_\bullet(A_\bullet(t)), \partial(t),  d(t))\mid  t\in \mathbb{R}\}
 \end{eqnarray}
    where  $\partial(t) =\sum_{i=0}^n(-1)^i  \partial_n^i(t)$  for  any  $t\in \mathbb{R}$.  
    \end{proposition}
    
    \begin{proof}  
   For  any  $t\in \mathbb{R}$,   we  have  an  induced   chain  map 
    \begin{eqnarray*}
    \partial_n^i(t):   (C_\bullet(A_n(t)),   d(t))\longrightarrow   (C_\bullet(A_{n-1}(t)),   d(t)).  
    \end{eqnarray*}
     By  an  analog   of      Lemma~\ref{le-250112-7}  or  \cite[Lemma~3.1]{conf1},   
    we  have  a  persistent  chain  map  
    \begin{eqnarray*}
       \boldsymbol{ \partial}_n^i:   (C_\bullet(\mathbf{A}_n),  \mathbf{d} )\longrightarrow  
       (C_\bullet(\mathbf{A}_{n-1}),  \mathbf{d} ) 
    \end{eqnarray*}
    where  
   $
    \boldsymbol{ \partial}_n^i= \{\partial_n^i(t)\mid  t\in \mathbb{R}\}$,  and       consequently  
       a  persistent  chain  map 
    \begin{eqnarray*}
    \boldsymbol{ \partial}=\sum_{i=0}^n(-1)^i   \boldsymbol{ \partial}_n^i:  
      (C_\bullet(\mathbf{A}_n),  \mathbf{d} )\longrightarrow 
        (C_\bullet(\mathbf{A}_{n-1}),  \mathbf{d} ) 
    \end{eqnarray*} 
    such     that  
    $ \mathbf{d}^2=0$,  $ \boldsymbol{\partial}^2=0$  and  $ \boldsymbol{\partial} \mathbf{ d} =    \mathbf{d} \boldsymbol{ \partial}$.  
Thus       (\ref{eq-3.lllq2})  is  a  persistent  double  complex. 
    \end{proof}

    The  next  corollary  follows  from  Definition~\ref{def-7.2},  Definition~\ref{def-7.3}  
 and  
    Proposition~\ref{pr-3.2.a1}.

    \begin{corollary}
    \label{pr-250112-1}
    Let  $\mathbf{A}_\bullet=  \{A_\bullet(t)\mid  t\in \mathbb{R}\}$  
    be  a  persistent  $\Delta$-CW-complex.  
   Then  we  have  a  persistent  chain  complex  
   \begin{eqnarray}\label{eq-3.lllq77}
(H^*(\mathbf{A}_\bullet), \boldsymbol{\partial}) =\{(H^*(A_\bullet(t)), \partial(t))\mid  t\in \mathbb{R}\}
 \end{eqnarray}
 where  $H^*(\mathbf{A}_n )$  is  the  persistent  cohomology  ring  of  
 $\mathbf{A}_n$  for any  $n\in \mathbb{N}$.   
    \end{corollary}
    
    \begin{proof}
    Apply  the  Hom  functor  and  the  cohomology  functor  subsequently  
    to  the  persistent  chain  complex 
    \begin{eqnarray*} 
   (C_\bullet(\mathbf{A}_n),  \mathbf{d} )=  \{ (C_\bullet(A_n(t)), d(t))\mid  t\in \mathbb{R}\}  
     \end{eqnarray*}
     for  any  $n\in \mathbb{N}$.   
    The  persistent  double  complex  
    (\ref{eq-3.lllq2})  
      induces   the  persistent  chain  complex  (\ref{eq-3.lllq77})  
     where  $H^*(\mathbf{A}_n )$  is  the  persistent  cohomology  module  of  
 $\mathbf{A}_n$  for any  $n\in \mathbb{N}$.   By  Definition~\ref{def-7.2},
   we  have  the  cup  product     
   \begin{eqnarray}\label{eq-250112-51}
  \smile:   H^{p}(A_n(t)) \times  H^{q} (A_n(t))\longrightarrow  H^{p+q}(A_n(t))
   \end{eqnarray}  
  which  makes    $H^*(A_n(t))$    to  be    a   cohomology    ring  for  any   $t\in \mathbb{R}$.  
 By       Definition~\ref{def-7.3},     for  any  $t_1\leq  t_2$,  
\begin{eqnarray*}
\alpha_n(t_1,t_2)^*:  H^*(A_n(t_2))\longrightarrow  H^*(A_n(t_1))
\end{eqnarray*}
  is  a  homomorphism  of  cohomology  rings  induced  by  (\ref{eq-3.9az1}).  
 That  is,  
 the  diagram  commutes 
 \begin{eqnarray}\label{eq-250112-52}
 \xymatrix{
 H^{p}(A_n(t_2)) \times  H^{q} (A_n(t_2))  \ar[d]_-{\alpha_n(t_1,t_2)^*} \ar[r]^-{ \smile}  &H^{p+q}(A_n(t_2))  \ar[d]^-{\alpha_n(t_1,t_2)^*}\\
  H^{p}(A_n(t_1)) \times  H^{q} (A_n(t_1))   \ar[r]^-{ \smile}  &H^{p+q}(A_n(t_1)).   
 }
\end{eqnarray} 
  Therefore,    $H^*(\mathbf{A}_n )$  in   (\ref{eq-3.lllq77})  
    is  a  persistent  cohomology  ring.  
    \end{proof}

   \subsection{Persistent  $\Delta$-manifolds}
   
\begin{definition}\label{def-8.2}\cite[Definition~1~(2),  Section~3.1]{conf1}
A  {\it  $\Delta$-manifold}     is  a    $\Delta$-set  $M_\bullet= (M_n)_{n\in\mathbb{N}}$
 where  all the  sets  $M_n$  are  smooth  manifolds  and  all the  face  maps 
 $\partial_n^i:  A_n\longrightarrow  A_{n-1}$   are smooth.  
   \end{definition}   

\begin{definition}\label{def-001a1}
A  {\it   persistent  $\Delta$-manifold}     is  a  family  of   $\Delta$-manifolds 
 \begin{eqnarray*}
  \mathbf{M}_\bullet= \{  M_\bullet(t)\mid   t\in  \mathbb{R}\},  
  \end{eqnarray*}  
 where  $M_\bullet(t)$  is  a  $\Delta$-manifold  with  face  maps  
   $\partial^i_n(t)$  for     $t\in  \mathbb{R}$,  together  with  a  family  of   maps  
   \begin{eqnarray}\label{eq-3.9by1}
    \alpha_n(t_1,t_2):   M_n(t_1)\longrightarrow  M_n(t_2)
    \end{eqnarray} 
    for     $n\in  \mathbb{N}$    and     $t_1\leq  t_2$  such  that  
    the  identities  in  
      Definition~\ref{def-0.za1}~(1), (2)  and   the  commutative  diagram  
      (\ref{eq-diag-3.a1})   are  satisfied
          for   any  $0\leq  i\leq  n$.    
    \end{definition}
    
    Let  $\mathbf{M}_\bullet=  \{M_\bullet(t)\mid  t\in \mathbb{R}\}$  
    be  a  persistent  $\Delta$-manifold,  where  $M_\bullet(t)= (M_n(t))_{n\in \mathbb{N}}$  
    is  a  $\Delta$-manifold  for  any  $t\in    \mathbb{R}$  with  face  maps  
    $\partial_n^i(t):  M_n(t)\longrightarrow  M_{n-1}(t)$   for  any  $n\in \mathbb{N}$  and  any  
    $0\leq  i\leq  n$.  
        Let  $(\Omega^\bullet(M_n(t)), d(t))$  be  the   de-Rham  cochain  complex  
    with 
    $\Omega^m(M_n(t))$  the  real  vector  space   spanned  by  the  $m$-forms  on  $M_n(t)$ 
    and  the  
     coboundary  maps 
    \begin{eqnarray*}
    d^m(t):  \Omega^m(M_n(t))\longrightarrow  \Omega^{m+1}(M_n(t))
    \end{eqnarray*}
      for  any  $m\in  \mathbb{N}$.  
    We  have  the  exterior   product  of  differential  forms  
    \begin{eqnarray}\label{eq-250112-85}
    \wedge:  \Omega^p(M_n(t))\times  \Omega^q(M_n(t))  \longrightarrow  \Omega^{p+q}(M_n(t)).  
    \end{eqnarray}

  \begin{proposition}\label{pr-3.2.b2}
  Let  $\mathbf{M}_\bullet=  \{M_\bullet(t)\mid  t\in \mathbb{R}\}$  
    be  a  persistent  $\Delta$-manifold.  Then   
    we  have  a    persistent  double  complex  
 \begin{eqnarray}\label{eq-3.lllq2aa1}
(\Omega^\bullet(\mathbf{M}_\bullet), \boldsymbol{\partial},  \mathbf{d}) =\{(\Omega^\bullet(M_\bullet(t)), \partial(t),  d(t))\mid  t\in \mathbb{R}\}
 \end{eqnarray}
 with  a  persistent  exterior    product  
 \footnote[3]{
 A {\it  persistent  exterior  product}  (\ref{eq-250112-81}) 
  is  a  family  of  exterior  products  
 (\ref{eq-250112-85})   for  $t\in \mathbb{R}$  
     such  that    for  any  $t_1\leq  t_2$,  
\begin{eqnarray*}
\alpha_n(t_1,t_2)^\#:  \Omega^\bullet(M_n(t_2))\longrightarrow   \Omega^\bullet(M_n(t_1))
\end{eqnarray*}
  is  a    homomorphism  of  algebras  induced  by  (\ref{eq-3.9by1}).  
   }
 \begin{eqnarray}\label{eq-250112-81}
 \wedge: \Omega^p(\mathbf{M}_n)\times  \Omega^q(\mathbf{M}_n) \longrightarrow  
 \Omega^{p+q}(\mathbf{M}_n)   
 \end{eqnarray}
    for  any  $n\in  \mathbb{N}$  
    where  $\partial(t) =\sum_{i=0}^n(-1)^i  \partial_n^i(t)$      for  any  $t\in \mathbb{R}$. 
    \end{proposition}
    
    \begin{proof} 
    By  an  analog    of      Lemma~\ref{le-250112-7}  or  \cite[Lemma~3.1]{conf1},   
    (\ref{eq-3.lllq2aa1})
   is  a  persistent  double  complex.       
    For  any  $t_1\leq  t_2$,  
    the  diagram  commutes
     \begin{eqnarray}\label{eq-250112-72}
 \xymatrix{
 \Omega^{p}(M_n(t_2)) \times  H^{q} (M_n(t_2))  \ar[d]_-{\alpha_n(t_1,t_2)^\#} \ar[r]^-{ \wedge}  &\Omega^{p+q}(M_n(t_2))  \ar[d]^-{\alpha_n(t_1,t_2)^\#}\\
   \Omega^{p}(M_n(t_1)) \times   \Omega^{q} (M_n(t_1))   \ar[r]^-{ \wedge}  & \Omega^{p+q}(M_n(t_1))    
 }
\end{eqnarray} 
    where  $\alpha_n(t_1,t_2)^\#$  is  induced  by   (\ref{eq-3.9by1}).   
    Thus  (\ref{eq-250112-81})  is  a  persistent  exterior  product.  
    \end{proof}

   \begin{corollary}
    \label{pr-250112-2ax}
    Let  $\mathbf{M}_\bullet=  \{M_\bullet(t)\mid  t\in \mathbb{R}\}$  
    be  a  persistent  $\Delta$-manifold.  
  Then  we  have  a  persistent  chain  complex  
   \begin{eqnarray}\label{eq-3.lllq887}
(H^*(\mathbf{M}_\bullet), \boldsymbol{\partial}) =\{(H^*(M_\bullet(t)), \partial(t))\mid  t\in \mathbb{R}\}
 \end{eqnarray}
 where  $H^*(\mathbf{M}_n )$  is  the  persistent   cohomology  ring  of  
 $\mathbf{M}_n$  for any  $n\in \mathbb{N}$.   
    \end{corollary}
    
    \begin{proof}
    The  corollary  is  an  analog   of  Corollary~\ref{pr-250112-1}.  
    In  particular,  if  we  take $H^*(\mathbf{M}_n )$  as  the  
    de-Rham  cohomology  with  coefficients  in  $\mathbb{R}$,  
    then  the  proof  follows  from  Proposition~\ref{pr-3.2.b2}.  
         \end{proof}

\section{Persistent  bundles  over  configuration  spaces}\label{25-sect4}

In this  section,  we  study  the  persistent  covering  maps  from   ordered  configuration  spaces
 to  unordered  configuration  spaces as  well  as  the  associated  persistent  vector  bundles.  
 We   study  the  configuration  spaces of  Riemannian  manifolds   in  Subsection~\ref{ss3.2}  and  study  the  configuration  spaces  of  graphs  in  Subsection~\ref{ss3.1}.

Let  $(X,d)$  be  a  metric  space  where  $d:  X\times  X\longrightarrow  [0,+\infty]$
is  a  distance.  
Let  $k$  be  a  positive  integer. 
Let  $r\geq  0$. 
The  {\it  $k$-th  ordered  configuration  space  of  hard spheres}  of  radius  $r$ of  $X$ is  
\begin{eqnarray}\label{eq-0.1}
{\rm  Conf}_k(X, r)= \{(x_1,\ldots,x_k)\in  X^k \mid  d(x_i,x_j)>2r {\rm~for~ }  i\neq  j\}
\end{eqnarray}   
   where  $X^k$  is  the  $k$-fold   Cartesian  product  of  $X$. 
   In  particular,  let  $r=0$  in  (\ref{eq-0.1}).  
    The  {\it  $k$-th  ordered  configuration  space}  of      $X$  is
    \begin{eqnarray}\label{eq-0.2}
    {\rm  Conf}_k(X) =     {\rm  Conf}_k(X,0)= \{(x_1,\ldots,x_k)\in  X^k \mid  x_i\neq  x_j{\rm~for~ }  i\neq  j\},  
    \end{eqnarray}
    which  only  depends  on the  topology  of  $X$  and  does  not  depend   on  the  choice  of 
    the  metric $d$ that  is  compatible  with  the  topology.  
Let  $\Sigma_k$  be  the   $k$-th  symmetric  group.  
Let $\Sigma_k$  act  on  $X^k$  by  permuting  the  coordinates  from  the  left
by  (\ref{eq-25-01-31-1}).  
 Then  $ {\rm  Conf}_k(X)$  is  an  open    subspace  of  $X^k$  consisting  of  the  points  with  trivial 
  isotropy groups.  
   We  have a  $\Sigma_k$-invariant  filtration  
  \begin{eqnarray}\label{eq-0.a2}
  {\rm  Conf}_k(X,-)= \{{\rm  Conf}_k(X,r)\mid  r\geq  0\}
  \end{eqnarray}
 of   $ {\rm  Conf}_k(X)$  such  that  
 \begin{enumerate}[(1)]
 \item
  ${\rm  Conf}_k(X,r)$  is  $\Sigma_k$-invariant  for  any  $r\geq  0$;  
 \item
  ${\rm  Conf}_k(X,r_1)\subseteq  {\rm  Conf}_k(X,r_2)$  for  any  $r_2\leq  r_1$;  
 \item
 if  $d(x,y)<+\infty$  for  any  $x,y\in  X$,  then  
  $\bigcap_{r\geq  0}{\rm  Conf}_k(X,r)=\emptyset$.  
 \end{enumerate}

 Let  $2^X$  be  the  power set  of  $X$  whose  elements  are  subsets  of  $X$.    
 The  {\it  $k$-th  unordered  configuration  space  of  hard spheres}  of  radius  $r$ of  $X$ is 
 the   $\Sigma_k$-orbit  space  
 \begin{eqnarray*}
 {\rm  Conf}_k(X, r)/\Sigma_k= \{\{x_1,\ldots,x_k\}\in  2^X  \mid    d(x_i,x_j)>2r {\rm~for~ }  i\neq  j\}
 \end{eqnarray*}
 and  the  {\it  $k$-th  unordered  configuration  space}  of      $X$  is
    \begin{eqnarray*} 
 {\rm  Conf}_k(X)/\Sigma_k =  {\rm  Conf}_k(X,0)/\Sigma_k= \{\{x_1,\ldots,x_k\}\in  2^X \mid  x_i\neq  x_j{\rm~for~ }  i\neq  j\}.    
    \end{eqnarray*}
    
    \begin{lemma}\label{le-2.aa1}
   For any  positive  integer  $k$,
    we  have  a  persistent  covering  map  
    \begin{eqnarray*}
     \pi(k,-):   {\rm  Conf}_k(X,-)\longrightarrow    {\rm  Conf}_k(X,-)/\Sigma_k.
    \end{eqnarray*}    
    \end{lemma}
    
    \begin{proof}
    Let $k$  be  any  positive  integer.   Let  $r\geq  0$. 
    Then    
   we  have  a   $k!$-sheeted  covering  map   
    \begin{eqnarray}\label{eq-0b.5}
    \pi(k,r):   {\rm  Conf}_k(X,r)\longrightarrow    {\rm  Conf}_k(X,r)/\Sigma_k.  
    \end{eqnarray}
    In  particular,  let  $r=0$.  We  
have  a   $k!$-sheeted  covering  map   
    \begin{eqnarray*}
    \pi(k)=\pi(k,0):   {\rm  Conf}_k(X)\longrightarrow    {\rm  Conf}_k(X)/\Sigma_k.  
    \end{eqnarray*}
    The $\Sigma_k$-invariant  filtration  (\ref{eq-0.a2})  of  $ {\rm  Conf}_k(X)$  
   induces  a    filtration  
     \begin{eqnarray}\label{eq-0b1}
  {\rm  Conf}_k(X,-)/\Sigma_k= \{{\rm  Conf}_k(X,r)/\Sigma_k\mid  r\geq  0\}
  \end{eqnarray}
   of  $  {\rm  Conf}_k(X)/\Sigma_k$
  such  that  for  any  $  r_2\leq  r_1$,  the  diagram  
  \begin{eqnarray*}
  \xymatrix{
 {\rm  Conf}_k(X,r_1) \ar[r] \ar[d]_{\pi(k,r_1)}  &{\rm  Conf}_k(X,r_2) \ar[d]^{\pi(k,r_2)}\\
{\rm  Conf}_k(X,r_1)/\Sigma_k   \ar[r] &{\rm  Conf}_k(X,r_2)/\Sigma_k
  }
  \end{eqnarray*}
  is  a  pull-back.  Here  the  horizontal  maps  are  canonical  inclusions.  
   Consequently,  
  we  have a     persistent covering  map   
  \begin{eqnarray*}
  \pi(k,-):    {\rm  Conf}_k(X,-)\longrightarrow    {\rm  Conf}_k(X,-)/\Sigma_k
  \end{eqnarray*}
  given  by 
  \begin{eqnarray*}
  \pi(k,-)=\{\pi(k,r)\mid  r\geq  0\}  
  \end{eqnarray*}
  such  that        
  $\pi(k,r)$  is  the  restriction  of  $\pi(k)$  to  ${\rm  Conf}_k(X,r)$  for  any  $r\geq  0$.      
    \end{proof}

 Since  $X$  is  a  CW-complex,  $X^k$  is  a  CW-complex.  
It  follows  that  both  $ {\rm  Conf}_k(X,r)$  and  $ {\rm  Conf}_k(X,r)/ \Sigma_k $  are  
CW-complexes.    
Let  $\Sigma_k$  act  on  $\mathbb{F}^k$  by  permuting  the  coordinates  
from  the  right  by  (\ref{eq-25-01-31-2}).  
Associated to  the  covering  map  $\pi(k,r)$  in  (\ref{eq-0b.5}),  
we  have  a  $O(\mathbb{F}^k)$-bundle  
\begin{eqnarray*}
\xi(X,k,r;\mathbb{F}):  \mathbb{F}^k\longrightarrow   
{\rm  Conf}_k(X,r)\times_{\Sigma_k}   \mathbb{F}^k
\longrightarrow  {\rm  Conf}_k(X,r)/ \Sigma_k.   
\end{eqnarray*}
Consider  the  $(k-1)$-dimensional  subspace
\begin{eqnarray*}
W=\Big\{(x_1,\cdots,x_k)\in \mathbb{F}^k~\big|~
  \sum_{i=1}^k x_i=0\Big\}
\end{eqnarray*}
    of  $\mathbb{F}^k$.  
    Since  $W$  is  $\Sigma_k$-invariant,  
 we   have  a  $O(\mathbb{F}^{k-1})$-bundle 
 \begin{eqnarray*}
\zeta (X,k,r; \mathbb{F}):  
W\longrightarrow  
{\rm  Conf}_k(X,r)\times_{\Sigma_k}W
\longrightarrow {\rm  Conf}_k(X,r)/\Sigma_k.   
\end{eqnarray*}

\begin{theorem}\label{pr-3.2aaz}
  For any  positive  integer  $k$,
    we  have  a   persistent   $O(\mathbb{F}^k)$-bundle  
    \begin{eqnarray}\label{eq-3.96}
 \boldsymbol{\xi}(X,k,-;\mathbb{F})=\{\xi(X,k,r;\mathbb{F}) \mid  r\geq  0\} 
\end{eqnarray}
and  a   persistent   $O(\mathbb{F}^{k-1})$-bundle  
    \begin{eqnarray}\label{eq-3.97}
 \boldsymbol{\zeta}(X,k,-;\mathbb{F})=\{\zeta(X,k,r;\mathbb{F}) \mid  r\geq  0\}    
\end{eqnarray}
over  ${\rm  Conf}_k(X,-)/\Sigma_k$  such that
\begin{eqnarray}\label{eq-02.77}
 \boldsymbol{\xi}(X,k,-;\mathbb{F}) \cong    \boldsymbol{\zeta}(X,k,-;\mathbb{F})\oplus   \boldsymbol{\epsilon}(X,k,-;\mathbb{F}) 
\end{eqnarray}
where  the  righthand  side  of  (\ref{eq-02.77})  is  the  persistent  Whitney  sum  of  
(\ref{eq-3.97})  and  the  persistent  trivial  $\mathbb{F}$-line  bundle
\begin{eqnarray}\label{eq-3.98}
\boldsymbol{\epsilon}(X,k,-;\mathbb{F}) =  \{\epsilon(X,k,r;\mathbb{F}) \mid  r\geq  0\} 
\end{eqnarray}
     over  ${\rm  Conf}_k(X,-)/\Sigma_k$.  
\end{theorem}

\begin{proof}
Let $k$  be  any  positive  integer.   Let  $r\geq  0$. 
    Then    we  have  a  Whitney  sum  of vector  bundles
\begin{eqnarray}\label{eq-02.2}
\xi(X,k,r;\mathbb{F}) \cong   \zeta(X,k,r;\mathbb{F})\oplus  \epsilon(X,k,r;\mathbb{F}) 
\end{eqnarray}
where  $\epsilon(X,k,r;\mathbb{F})$  is  the  trivial  
$\mathbb{F}$-line  bundle  over  
$  {\rm  Conf}_k(X,r)/\Sigma_k  $.  
Let    $r_2\leq  r_1$.  
With  the  help  of  Lemma~\ref{le-2.aa1},  we  have  pull-backs  of  vector  bundles
\begin{eqnarray*}
\xi(X,k,r_1;\mathbb{F})  &=&  \iota_{r_1,r_2}^* (\xi(X,k,r_2;\mathbb{F}) ),\\
\zeta(X,k,r_1;\mathbb{F})  &=&  \iota_{r_1,r_2}^* (\zeta(X,k,r_2;\mathbb{F}) ),\\
\epsilon(X,k,r_1;\mathbb{F})  &=&  \iota_{r_1,r_2}^* (\epsilon(X,k,r_2;\mathbb{F}) ) 
\end{eqnarray*}
where  
$\iota_{r_1,r_2}:  {\rm  Conf}_k(X,r_1)/\Sigma_k \longrightarrow {\rm  Conf}_k(X,r_2)/\Sigma_k $
is  the  canonical  inclusion.  
Consequently,  we  obtain  the  persistent  vector  bundles  (\ref{eq-3.96}),  
(\ref{eq-3.97})  and  (\ref{eq-3.98})   such  that  
  (\ref{eq-02.77})  is  satisfied.  
\end{proof}

\begin{corollary}
For any  positive  integer  $k$,  
\begin{eqnarray}\label{eq-3.b81}
w( \boldsymbol{\xi}(X,k,-;\mathbb{R}) )= w( \boldsymbol{\zeta}(X,k,-;\mathbb{R}) ).  
\end{eqnarray}
\end{corollary}

\begin{proof}
   Take  the  persistent  cohomology  ring  
   \begin{eqnarray*}
   H^*(  {\rm  Conf}_k(X,-)/\Sigma_k;\mathbb{Z}_2) 
   \end{eqnarray*}  
      of  $  {\rm  Conf}_k(X,-)/\Sigma_k$   with  coefficients in  $\mathbb{Z}_2$.  
 By   (\ref{eq-sw.1})  and  (\ref{eq-02.77}),  
  \begin{eqnarray*}
  w( \boldsymbol{\xi}(X,k,-;\mathbb{R}) )& =&   w(  \boldsymbol{\zeta}(X,k,-;\mathbb{R})\oplus  \boldsymbol{\epsilon}(X,k,-;\mathbb{R}) )\\
  &=&  w(  \boldsymbol{\zeta}(X,k,-;\mathbb{R}))  w(  \boldsymbol{\epsilon}(X,k,-;\mathbb{R}) )\\
  &=&   w( \boldsymbol{\zeta}(X,k,-;\mathbb{R})).   
  \end{eqnarray*}
  We  obtain  (\ref{eq-3.b81}).  
\end{proof}

\begin{corollary}
For any  positive  integer  $k$,
\begin{eqnarray}\label{eq-3.b82}
c( \boldsymbol{\xi}(X,k,-;\mathbb{C}) )= c( \boldsymbol{\zeta}(X,k,-;\mathbb{C}) ).  
\end{eqnarray}
\end{corollary}

\begin{proof}
  Take  the  persistent  cohomology  ring  
   \begin{eqnarray*}
   H^*(  {\rm  Conf}_k(X,-)/\Sigma_k;\mathbb{Z}) 
   \end{eqnarray*}  
      of  $  {\rm  Conf}_k(X,-)/\Sigma_k$   with  coefficients in  $\mathbb{Z}$.  
 By   (\ref{eq-c.1})  and  (\ref{eq-02.77}),  
  \begin{eqnarray*}
  c( \boldsymbol{\xi}(X,k,-;\mathbb{C}) )& =&   c(  \boldsymbol{\zeta}(X,k,-;\mathbb{C})\oplus  \boldsymbol{\epsilon}(X,k,-;\mathbb{C}) )\\
  &=&  c(  \boldsymbol{\zeta}(X,k,-;\mathbb{C}))  c(  \boldsymbol{\epsilon}(X,k,-;\mathbb{C}) )\\
  &=&   c( \boldsymbol{\zeta}(X,k,-;\mathbb{C})).   
  \end{eqnarray*}
  We  obtain  (\ref{eq-3.b82}).  
\end{proof}

We  have  a  persistent  $\Delta$-CW-complex  
\begin{eqnarray}\label{eq-5mboa97}
{\rm  Conf}_\bullet(X, -)  =  ({\rm  Conf}_k(X,-))_{k\geq  1} 
 = \{{\rm  Conf}_\bullet(X, r)\mid  r\geq  0\}   
\end{eqnarray}
 with   face  maps  
      \begin{eqnarray}\label{eq-2501-991}
      \partial_k^i(r):  {\rm  Conf}_{k} (X,r) \longrightarrow  {\rm  Conf}_{k-1} (X,r)
      \end{eqnarray}
       for  any  $r\geq  0$,   any  positive  integer  $k$   and  any  $0\leq  i\leq  k$    
      given  by    
      \begin{eqnarray*}
      \partial_k^i(r):   (x_1,\ldots, x_k)\mapsto        (x_1,\ldots,  \widehat{x_i},\ldots, x_k).  
      \end{eqnarray*}
                The  next  theorem  follows  from  Proposition~\ref{pr-3.2.a1}  
      and  Corollary~\ref{pr-250112-1}. 

 \begin{theorem}\label{th-25-5.1.1}
   We  have  a  persistent  double  complex  
 \begin{eqnarray}\label{eq-5.lllq2}
(C_\bullet(  { \rm  Conf}_\bullet (X,-)), \boldsymbol{\partial},  \mathbf{d}) 
 \end{eqnarray}
 and  a  persistent  chain  complex  
   \begin{eqnarray}\label{eq-5.lllq77}
(H^*( { \rm  Conf}_\bullet (X,-)), \boldsymbol{\partial}) 
 \end{eqnarray}
 where  $H^*({ \rm  Conf}_k (X,-) )$  is  the  persistent  cohomology  ring  of  
 ${ \rm  Conf}_k (X,-) $  for any  $k\geq  1$.   
 Here  $\partial(r) =\sum_{i=0}^k(-1)^i  \partial_k^i(r)$ 
    and  $d(r)$  is  the  cellular  boundary  map  of  $ { \rm  Conf}_\bullet (X,r)$ 
     for  any  $r\geq  0$. 
     \end{theorem}
     
     \begin{proof}
     We  have  a persistent  $\Delta$-CW-complex
     (\ref{eq-5mboa97}).
                By  the  Cellular  Approximation  Theorem,  we  can  choose  appropriate  
      cell  structures  of   $ {\rm  Conf}_\bullet (X,r)$
      such  that  all       the  boundary  maps   (\ref{eq-2501-991})  are  cellular.  
      The  diagram  commutes 
      \begin{eqnarray*}
      \xymatrix{
      {\rm  Conf}_{k} (X,r_1)  \ar[rr]^-{\partial_k^i(r_1)}   \ar[d] &&{\rm  Conf}_{k-1} (X,r_1)  \ar[d]\\ 
      {\rm  Conf}_{k} (X,r_2)  \ar[rr]^-{\partial_k^i(r_2)}   &&{\rm  Conf}_{k-1} (X,r_2)
      }
      \end{eqnarray*}
      for  any  $r_2\leq  r_1  $,  where  
      the  vertical  maps  are  canonical  inclusions.  
      By   Proposition~\ref{pr-3.2.a1},  we  obtain  the    
          persistent  double  complex  
      (\ref{eq-5.lllq2}).  
      By  Corollary~\ref{pr-250112-1},  
     we  obtain  the   persistent  chain  complex  
     (\ref{eq-5.lllq77}).
        \end{proof}

Consider  the  $\Sigma_k$-action  
on  ${\rm  Conf}_k(X,-)$  for  any  $k\geq  1$.  
This  induces  a  $\Sigma_k$-action  
on  the  persistent  cohomology  ring  $H^*( {\rm  Conf}_k(X,-))$.  
The  next  corollary  follows  from  Theorem~\ref{th-25-5.1.1}.

\begin{corollary}\label{co-25-01-15-1}
The  persistent  chain  complex     (\ref{eq-5.lllq77})   has  a    $\Sigma_\bullet$-action  
such  that  the  diagram   commutes 
\begin{eqnarray}\label{eq-250116-91}
\xymatrix{
 H^*( { \rm  Conf}_k (X,r_2))\ar[d]  \ar[r]^-{\sigma^* }  & H^*( { \rm  Conf}_k (X,r_2))\ar[d]\\
  H^*( { \rm  Conf}_k (X,r_1))\ar[r]^-{\sigma^* }& H^*( { \rm  Conf}_k (X,r_1))
}
\end{eqnarray} 
for  any  $r_1\geq  r_2\geq  0$,  any  positive  integer  $k$  and  any  $\sigma\in \Sigma_k$.  
Here  the  vertical  maps  are  induced  by  canonical  inclusions  of  
$ { \rm  Conf}_k (X,r_1)$  into  $ { \rm  Conf}_k (X,r_2)$.  
\end{corollary}

\begin{proof}
The     diagram  commutes 
\begin{eqnarray*}
\xymatrix{
  { \rm  Conf}_k (X,r_1) \ar[d]  \ar[r]^-{\sigma }  &   { \rm  Conf}_k (X,r_1) \ar[d]\\
   { \rm  Conf}_k (X,r_2) \ar[r]^-{\sigma  }&   { \rm  Conf}_k (X,r_2) 
}
\end{eqnarray*} 
where  the  vertical  maps  are   canonical  inclusions  of  
$ { \rm  Conf}_k (X,r_1)$  into  $ { \rm  Conf}_k (X,r_2)$.
Taking   the  cohomology  rings,  
we  obtain    the  commutative  diagram  (\ref{eq-250116-91}).  
\end{proof}

We  have  the  following  observations.  
\begin{enumerate}[(1)]
\item
   If $X$  and  $X'$   are  homeomorphic,  then  ${\rm  Conf}_k(X)$  and   
    ${\rm  Conf}_k(X')$ are  homeomorphic  thus  the  vector  bundles     
    $ \xi(X,k,0;\mathbb{F})$  and  $\xi(X',k,0;\mathbb{F})$
    are  isomorphic; 
    \item  
     If $(X,d)$  and  $(X',d)$   are  isometric,  then   ${\rm  Conf}_k(X,-)$  and   
    ${\rm  Conf}_k(X',-)$ are  isometric  thus the 
      persistent  vector  bundles   
  $\boldsymbol{ \xi}(X,k,-;\mathbb{F})$  and  $\boldsymbol{\xi}(X',k,-;\mathbb{F})$
 are   isomorphic  (cf.   Definition~\ref{def-25-01-97}).  
 \end{enumerate}

\subsection{Configuration  spaces of  Riemannian  manifolds}\label{ss3.2}

Let  $M$  be  a  Riemannian   manifold. 
Let  $d:  M\times  M\longrightarrow  [0,+\infty]$  be  the  distance  on  $M$
such  that  
\begin{eqnarray}\label{eq-dist-riem}
d(p,q)=
\inf_{\gamma:[0,1]\longrightarrow  M\atop \gamma(0)=p{\rm~and~}\gamma(1)=q} {\rm  length}(\gamma) 
\end{eqnarray}
 with 
 \begin{eqnarray*}
{\rm  length}(\gamma)=\int_0^1  |\gamma'(t)|dt   
\end{eqnarray*} 
if  there      exists  a   path $\gamma$  from  $p$  to  $q$
and   $d(p,q)=+\infty$  otherwise.  
 If  $M$  is     complete,  then  by  the  Hopf-Rinow Theorem  (cf.  \cite[p.  138]{oneill}),  
 for  any  
 $p,q\in  M$  with  $d(p,q)<+\infty$,  there  exists  a  geodesic  $\gamma$  on  $M$ 
  such  that  
 $\gamma(0)=p$,  $\gamma(1)=q$  and  ${\rm  length}(\gamma)=d(p,q)$.

Let  $r\geq  0$.    The   $k$-th  ordered  configuration  space  of   $M$ 
of  hard  spheres  with  radius  $r $  is    the  
Riemannian  
submanifold  
\begin{eqnarray}\label{eq-250217a}
{\rm  Conf}_{k} (M,r)=\Big\{(p_1,\ldots,p_k)\in  \prod_k  M
~\big|~
d_M(p_i,p_j)>2r{\rm~for~any~}   i \neq  j\Big\}
\end{eqnarray}
 of  the  product  manifold   $\prod_k  M$. 
 Let  $g$  be the  Riemannian  metric   on  $M$.  
 Then  (\ref{eq-dist-riem})  is  induced  from  $g$.  
 The  $k$-fold  product  metric  $g^k$
   is  a  Riemannian  metric  on  $\prod_k  M$  
 hence  it   induces  a  Riemannian  metric  
 $g^k\mid  _{{\rm  Conf}_{k} (M,r)}$  on  ${\rm  Conf}_{k} (M,r)$.

 \begin{lemma}\label{le-250208}
 For  any  Riemannian  manifold  $M$,  
the  symmetric  group  $\Sigma_k$  acts  on  ${\rm  Conf}_{k} (M,r)$  isometrically.  That  is,  
 for  any  $\sigma\in \Sigma_k$,  the  
 induced  pull-back 
 \begin{eqnarray*}
 \sigma^*: ( T^*{\rm  Conf}_{k} (M,r))^{\otimes  2}\longrightarrow  (T^*{\rm  Conf}_{k} (M,r))^{\otimes  2}
 \end{eqnarray*}
   of  covariant  tensor  fields  on  ${\rm  Conf}_{k} (M,r)$  satisfies  
 \begin{eqnarray}\label{eq-250109-7a}
\sigma^*(g^k\mid  _{{\rm  Conf}_{k} (M,r)}) = g^k\mid  _{{\rm  Conf}_{k} (M,r)}.  
 \end{eqnarray}  
 \end{lemma}
 
 \begin{proof} 
Take  any   smooth  vector  fields  
  $(V_1,\ldots,V_k),  (U_1,\ldots,U_k)\in    \Gamma(T {\rm  Conf}_{k} (M,r))$  
on the  $k$-th  ordered  configuration  space  (\ref{eq-250217a})   of  hard  spheres  with  
radius  $r$  
 where  $V_i,U_i\in  \Gamma(TM)$  are  smooth  vector  fields  on  $M$  for  $1\leq  i\leq  k$.  
 Then  
 \begin{eqnarray*}
 \sigma^*(g^k)  ((V_1,\ldots,V_k),  (U_1,\ldots,U_k) ) 
 &=& g^k  (\sigma_*(V_1,\ldots,V_k), \sigma_* (U_1,\ldots,U_k) ) \\
 &=&  g^k((V_{\sigma(1)}, \ldots,  V_{\sigma(k)}),  (U_{\sigma(1)}, \ldots,  U_{\sigma(k)}))\\
 &=& \prod_{i=1}^k  g(V_{\sigma(i)},  U_{\sigma(i)})\\
 &=&  \prod_{i=1}^k  g(V_{i},  U_{i})\\
 &=& g^k   ((V_1,\ldots,V_k),  (U_1,\ldots,U_k) ).   
 \end{eqnarray*}
 Here  $\sigma_*$  is  the  tangent  map  of  $\sigma$,  which  
 is  the    permutation  induced  by  $\sigma$  
 on  the  coordinates  of  the  tangent  vector  fields.  
 We  obtain  (\ref{eq-250109-7a}).  
  \end{proof}
  
   The  next  corollary  follows  from  Lemma~\ref{le-250208}.  
 
 \begin{corollary}\label{co-250208-1}  
  For  any  Riemannian  manifold  $M$,  
  the     $\Sigma_k$-action   on   ${\rm  Conf}_{k} (M,r)$  
  induces  an     
    orbit   manifold 
 \begin{eqnarray}\label{eq-250209-8b}
{\rm  Conf}_{k} (M,r)/\Sigma_k=\Big\{\{p_1,\ldots,p_k\} ~\big|~
d_M(p_i,p_j)>2r{\rm~for~any~}   i \neq  j\Big\}  
\end{eqnarray}
with  the  Riemannian  metric  
$(g^k\mid  _{{\rm  Conf}_{k} (M,r)})/\Sigma_k$ 
such  that  
\begin{eqnarray}\label{eq-250208-pull-back-g}
 g^k\mid  _{{\rm  Conf}_{k} (M,r)} =\pi(k,r)^*\big((g^k\mid  _{{\rm  Conf}_{k} (M,r)})/\Sigma_k\big).  
\end{eqnarray}
\end{corollary}

\begin{proof}
Let  $\{p_1,\ldots,p_k\}\in  {\rm  Conf}_{k} (M,r)/\Sigma_k$.  
Take  tangent  vectors       
\begin{eqnarray*}
X_{\{p_1,\ldots,p_k\}},  Y_{\{p_1,\ldots,p_k\}}  \in 
 T_{\{p_1,\ldots,p_k\}}\big({\rm  Conf}_{k} (M,r)/\Sigma_k\big). 
\end{eqnarray*} 
Define  
\begin{eqnarray}\label{eq-250209-permute}
\big((g^k\mid  _{{\rm  Conf}_{k} (M,r)})/\Sigma_k\big)(X_{\{p_1,\ldots,p_k\}},  Y_{\{p_1,\ldots,p_k\}})
= g^k(\tilde  X_{(p_{\sigma(1)},\ldots,p_{\sigma(k)})},  \tilde Y_{(p_{\sigma(1)},\ldots,p_{\sigma(k)})})
\end{eqnarray}
where  $\sigma\in \Sigma_k$,  $(p_{\sigma(1)},\ldots,p_{\sigma(k)}) \in   {\rm  Conf}_{k} (M,r)  $    and  
\begin{eqnarray*}
\tilde  X_{(p_{\sigma(1)},\ldots,p_{\sigma(k)})},  \tilde Y_{(p_{\sigma(1)},\ldots,p_{\sigma(k)})}  \in 
 T_{(p_{\sigma(1)},\ldots,p_{\sigma(k)})} {\rm  Conf}_{k} (M,r)  
\end{eqnarray*} 
such  that  
\begin{eqnarray*}
\pi(k,r)(p_{\sigma(1)},\ldots,p_{\sigma(k)}) &=&\{p_1,\ldots,p_k\},  \\
\pi(k,r)_*(\tilde  X_{(p_{\sigma(1)},\ldots,p_{\sigma(k)})})&=&X_{\{p_1,\ldots,p_k\}}, \\
\pi(k,r)_*(\tilde  Y_{(p_{\sigma(1)},\ldots,p_{\sigma(k)})})&=&Y_{\{p_1,\ldots,p_k\}}.  
\end{eqnarray*}
By  Lemma~\ref{le-250208},  (\ref{eq-250209-permute})  does  not  
depend  on  the  choice  of  $\sigma$.  
Hence  (\ref{eq-250209-permute})  is  well-defined.  
Consequently,   we  obtain (\ref{eq-250208-pull-back-g})  
from    (\ref{eq-250209-permute}).       
\end{proof}

By  Corollary~\ref{co-250208-1},  
  the  $k$-th     
  unordered  configuration  space  (\ref{eq-250209-8b})   of   $M$ 
of  hard  spheres  with  radius  $r $       is 
a  Riemannian  manifold.   
  The  next  corollary  follows  from  Lemma~\ref{le-2.aa1}  and  Proposition~\ref{pr-3.2aaz}.

\begin{corollary}\label{co-250118-1}
For  any  Riemannian  manifold  $M$,  
we  have  a    persistent   covering  map  of  Riemannian  manifolds 
\begin{eqnarray}\label{eq-aabnm12}
\pi(k,-):   {\rm  Conf}_{k} (M,-)\longrightarrow {\rm  Conf}_{k} (M,-)/\Sigma_k,   
\end{eqnarray}
which  is  locally  an  isometry,    
and  persistent  vector  bundles  
    \begin{eqnarray}\label{eq-03.88917}
 \boldsymbol{\xi}(M,k,-;\mathbb{F}) \cong    \boldsymbol{\zeta}(M,k,-;\mathbb{F})\oplus   \boldsymbol{\epsilon}(M,k,-;\mathbb{F}).  
\end{eqnarray}
\end{corollary}

\begin{proof}

For  any  $r\geq  0$,   
by  Corollary~\ref{co-250208-1},    
we  have  an  isometric   $k!$-sheeted  covering  map  
\begin{eqnarray*}
\pi(k,r):   {\rm  Conf}_{k} (M,r)\longrightarrow {\rm  Conf}_{k} (M,r)/\Sigma_k   
\end{eqnarray*}
such  that  the  Riemannian  metric  on  ${\rm  Conf}_{k} (M,r)$
is  the  pull-back  of   the  Riemannian  metric  on  ${\rm  Conf}_{k} (M,r)/\Sigma_k $. 
  Let  $X$  be  $M$  in  Lemma~\ref{le-2.aa1}.  We  have  the  persistent  covering  map 
   (\ref{eq-aabnm12}).  
   Let  $X$  be  $M$  in  Proposition~\ref{pr-3.2aaz}.   We  have  the persistent  vector  bundles  and  the  persistent   
  Whitney  sum  (\ref{eq-03.88917}). 
\end{proof}

           The  next  corollary  follows  from  Theorem~\ref{th-25-5.1.1}.

  \begin{corollary}\label{th-25-4.1.1}
     For  any  Riemannian  manifold  $M$,  
      we  have  a  persistent  double  complex  
 \begin{eqnarray}\label{eq-5.lllmm2}
(\Omega^\bullet(  { \rm  Conf}_\bullet (M,-)), \boldsymbol{\partial},  \mathbf{d}), 
 \end{eqnarray}
 where  $\Omega^\bullet(-)$  is  the  CDGA  of  differential  forms,  
 and  a  persistent  chain  complex  
   \begin{eqnarray}\label{eq-5.mmmq77}
(H^*( { \rm  Conf}_\bullet (M,-)), \boldsymbol{\partial}), 
 \end{eqnarray}
 where  $H^*(- )$  is  the  de-Rham  persistent  cohomology  algebra.   
     \end{corollary}
     
     \begin{proof}
     Let  $X$  be  $M$  in  Theorem~\ref{th-25-5.1.1}.  
     For  any  $r\geq  0$,  let  $d(r)$  be  the  exterior  derivative  of  
     $\Omega^\bullet ({ \rm  Conf}_\bullet (M,r))$.  
     Then   (\ref{eq-5.lllq2})  implies   (\ref{eq-5.lllmm2})  and  
     (\ref{eq-5.lllq77})  implies  (\ref{eq-5.mmmq77}).  
     \end{proof}

     The  next  corollary  follows   from  Corollary~\ref{co-25-01-15-1}. 
     
     \begin{corollary}\label{co-250118-91}
      For  any  Riemannian  manifold  $M$,  
      the   persistent  chain  complex     (\ref{eq-5.mmmq77})  of  de-Rham 
      persistent  cohomology   has  a    $\Sigma_\bullet$-action  
such  that  the  diagram   commutes 
\begin{eqnarray}\label{eq-250118-9998}
\xymatrix{
 H^*( { \rm  Conf}_k (M,r_2))\ar[d]  \ar[r]^-{\sigma^* }  & H^*( { \rm  Conf}_k (M,r_2))\ar[d]\\
  H^*( { \rm  Conf}_k (M,r_1))\ar[r]^-{\sigma^* }& H^*( { \rm  Conf}_k (M,r_1))
}
\end{eqnarray} 
for  any  $r_1\geq  r_2\geq  0$  and  any  $\sigma\in \Sigma_k$.  
Moreover,   
  the  de-Rham  persistent  cohomology  algebra 
\begin{eqnarray}\label{eq-250118-mn12g}
H^*({ \rm  Conf}_k (M,-)/\Sigma_k) 
\end{eqnarray}
is  the    $\Sigma_k$-invariant  sub-algebra of  $H^*({ \rm  Conf}_k (M,-) )$. 
     \end{corollary}
     
     \begin{proof}
     Let  $X$  be  $M$  in Corollary~\ref{co-25-01-15-1}.   
     Then  (\ref{eq-250116-91})  implies  (\ref{eq-250118-9998}).  
     Moreover,  the  persistent  covering  map  (\ref{eq-aabnm12})   induces  
     a  persistent  pull-back  of  differential  forms  
     \begin{eqnarray}\label{eq-250118-7c}
    \pi(k,-)^*:  \Omega^\bullet( {\rm  Conf}_{k} (M,-)/\Sigma_k)\longrightarrow  \Omega^\bullet({\rm  Conf}_{k} (M,-)),           
     \end{eqnarray}
     which  is  a  persistent  chain  map  of  persistent  chain complexes
     with  respect  to  the  exterior  derivative.  
     It  follows  that  (\ref{eq-250118-7c})  is  injective  and    the  image  of  (\ref{eq-250118-7c})  
     is  the  $\Sigma_k$-invariant   forms.  
     Applying  the  de-Rham  persistent  cohomology  functor  to   (\ref{eq-250118-7c}),  
     we  have a  homomorphism  of  persistent  cohomology algebra
      \begin{eqnarray}\label{eq-250118-9d}
    \pi(k,-)^*:  H^*( {\rm  Conf}_{k} (M,-)/\Sigma_k)\longrightarrow  H^*({\rm  Conf}_{k} (M,-)).            
     \end{eqnarray}
     Therefore,  
     (\ref{eq-250118-mn12g})  is  the    $\Sigma_k$-invariant  sub-algebra of  
     $H^*({ \rm  Conf}_k (M,-) )$. 
     \end{proof}

\subsection{Configuration  spaces of  graphs}\label{ss3.1}

Let  $V_G$  be  a  finite  or  countable  
discrete set.  The  elements  in  $V_G$  are  {\it  vertices}.  
Let  $V_G^2/  \mathbb{Z}_2$  
be  the  symmetric  product  of  $V_G$  such  that 
 $(u,v)\sim (v,u)$
for  any  $u,v\in  V_G$.  
Let $E_G$  be   a  subset  of  $V_G^2/  \mathbb{Z}_2$.   
The  elements  in  $E_G$  are  {\it  edges}.  
Then   $G=(V_G,  E_G)$  is   a  graph.      
Let  $n\in \mathbb{N}$.   
A   {\it  path}  of  length  $n$  in  $G$  is  a  sequence  
$v_0v_1\cdots  v_n$  such  that 
$v_i\in  V_G$  for  $0\leq  i\leq  n$  
and  $(v_{j-1},v_j)\in  E_G$  for  $1\leq  j\leq  n$.  
For  any  $u,v\in  V_G$  
such   that  $u\neq  v$,  
let   $d_G(u,v)$  be  the  minimum  length  
of  the  paths   connecting  $u$  and  $v$. 
If  there  does  not  exist  any  path   in  $G$  connecting  $u$  and  $v$,  
then  we  let    $d_G(u,v)$    be  $+\infty$.   
Then  $(V_G,d_G)$  is  a  metric  space.

Let  $k$ be  a  positive  integer.  
The    $k$-th  ordered  configuration  space  of   $V_G$ 
of  hard  spheres  with  radius  $n/2$  is  the  $k$-uniform  
  hyperdigraph  
  (for  example,  
  refer  to  \cite[Subsection~3.1]{hdg}  for hyperdigraphs)  
   whose  elements  are  (ordered)   $k$-sequences  of  vertices
\begin{eqnarray}\label{eq-250209-graph-1}
{\rm  Conf}_{k} (V_G,\frac{n}{2})=\Big\{(v_1,\ldots,v_k)\in  \prod_k  V_G
~\big|~
d_G(v_i,v_j)>n{\rm~for~any~}   i \neq  j\Big\}. 
\end{eqnarray}
The      $k$-th  unordered  configuration  space  of   $V_G$ 
of  hard  spheres  with  radius  $n/2$  
 is   the  $k$-uniform  hypergraph  (for  example,  
  refer  to  \cite{berge}  for hypergraphs)  whose  elements  are   
 (unordered)  $k$-hyperedges   of  vertices 
 \begin{eqnarray}\label{eq-250209-graph-2}
{\rm  Conf}_{k} (V_G,\frac{n}{2})/\Sigma_k=\Big\{\{v_1,\ldots,v_k\} ~\big|~
d_G(v_i,v_j)>n{\rm~for~any~}   i \neq  j\Big\}. 
\end{eqnarray}
Note  that  ${\rm  Conf}_{k} (V_G, 0)$ 
is  the  complete  $k$-uniform  hyperdigraph  on  $V_G$
and 
  ${\rm  Conf}_{k} (V_G, 0)/\Sigma_k$ 
is  the  complete  $k$-uniform  hypergraph  on  $V_G$,  both  of  which  
do   not  depend  on   the  choices  of   $E_G$.     
Hence without  loss  of  generality,  we  may  assume  $n\geq  1$ 
in  (\ref{eq-250209-graph-1})  and  (\ref{eq-250209-graph-2}).    
 The  family  of  hyperdigraphs 
  \begin{eqnarray}\label{eq-2.00}
  {\rm  Conf}_{k} (V_G,- ) = \Big\{{\rm  Conf}_{k} (V_G,\frac{n}{2}) ~\big| ~ n=1,2,\ldots
  \Big\}
  \end{eqnarray}
    is  a  $\Sigma_k$-invariant  filtration  of  $ {\rm  Conf}_{k} (V_G,1/2) $
    such that  
  \begin{eqnarray*}
  {\rm  Conf}_{k} (V_G,\frac{n}{2})  \supseteq {\rm  Conf}_{k} (V_G,\frac{n'}{2}) 
  \end{eqnarray*}
 for  any  $n\leq  n'$.     
   The  family  of  hypergraphs  
      \begin{eqnarray}\label{eq-2.1}
  {\rm  Conf}_{k} (V_G,- )/\Sigma_k=\Big \{{\rm  Conf}_{k} (V_G,\frac{n}{2})/\Sigma_k
  ~\big|~  n=1,2,\ldots\Big\}
  \end{eqnarray}
  is  a  filtration  of  $ {\rm  Conf}_{k} (V_G,1/2)/\Sigma_k$   induced  from  
  (\ref{eq-2.00}).

    An  {\it  independent set}  $S\subseteq  V_G$   of  $G$  
    is  a    set    of  vertices   such that   
    $d_G(u,v)>1$   for  any  $u,v\in  S$ with  $u\neq  v$.  
 The  {\it  independence  complex}  ${\rm  Ind}(G)$  of  $G$ 
 is  the  simplicial  complex  with  vertices  in  $V_G$  
 and  with  simplices  the  finite   independent sets  of  $G$
  (for  example,  refer  to  \cite{ind}  for  independence  complexes).  
 We  observe   that  
 ${\rm  Conf}_{k} (V_G,1/2)/\Sigma_k$  is  the  collection  
 of  all  the  $(k-1)$-simplices  (i.e. $k$-hyperedges)   in   ${\rm  Ind}(G)$.  
 Hence  
  \begin{eqnarray*}
  {\rm  Ind}(G) =  \bigcup_{k=1}^ \infty  
 {\rm  Conf}_{k} (V_G,\frac{1}{2})/\Sigma_k
 \end{eqnarray*}
 and  
 \begin{eqnarray*}
 {\rm  sk}^{l-1}  ({\rm  Ind}(G))=  \bigcup_{k=1}^l 
 {\rm  Conf}_{k} (V_G,\frac{1}{2})/\Sigma_k
 \end{eqnarray*}
 for  any  $l=1,2,\ldots$,   
 where    ${\rm  sk}^{l-1}  ({\rm  Ind}(G))$  
 is  the  $(l-1)$-skeleton  of  ${\rm  Ind}(G)$.  
 In  general,  for  any  $n=1,2,\ldots$, 
   we  consider  
 the  simplicial  complex   
  \begin{eqnarray*}
   {\rm  Ind}(G,\frac{n}{2}) =  \bigcup_{k=1}^ \infty  
 {\rm  Conf}_{k} (V_G,\frac{n}{2})/\Sigma_k   
 \end{eqnarray*}
 where ${\rm  Conf}_{k} (V_G,n/2)/\Sigma_k$  is  the  collection  
 of  all  the  $(k-1)$-simplices  in   ${\rm  Ind}(G,n/2)$,  i.e.
  \begin{eqnarray*}
 {\rm  sk}^{l-1}  ({\rm  Ind}(G,\frac{n}{2}))=  \bigcup_{k=1}^l 
 {\rm  Conf}_{k} (V_G,\frac{n}{2})/\Sigma_k
 \end{eqnarray*}   
 for  any  $l=1,2,\ldots$.  
 Then   
 \begin{eqnarray*}
  {\rm  Ind}(G) = {\rm  Ind}(G,\frac{1}{2}).  
  \end{eqnarray*}
  By   (\ref{eq-2.1}),  
 the  family   
  \begin{eqnarray*}
{\rm  Ind}(G,-) = \{{\rm  Ind}(G,\frac{n}{2}) \mid  n=1,2,\ldots\}
  \end{eqnarray*}
  is  a  filtration  of  ${\rm  Ind}(G) $  such  that
  $ {\rm  Ind}(G,n'/2)$  is  a  simplicial  sub-complex  of   
 $ {\rm  Ind}(G,n/2)$  
  for  any  $n\leq  n'$ 
  and  
    \begin{eqnarray*}
 {\rm  sk}^{l-1}  ({\rm  Ind}(G,\frac{n}{2}))=  {\rm  Ind}(G,\frac{n}{2})   \cap  {\rm  sk}^{l-1}  ({\rm  Ind}(G))  
  \end{eqnarray*}
  for  any  $l=1,2,\ldots$.  
     It  follows  that  
     \begin{eqnarray*}
     {\rm  sk}^{l-1}  ({\rm  Ind}(G,-))=  \Big\{ {\rm  sk}^{l-1}  ({\rm  Ind}(G,\frac{n}{2}))\mid   n-1,2,\ldots 
     \Big\}
     \end{eqnarray*}
     is  a  filtration  of  $ {\rm  sk}^{l-1}  ({\rm  Ind}(G))  $.  
 The  next  corollary  follows  from  Lemma~\ref{le-2.aa1}  and  Proposition~\ref{pr-3.2aaz}.

  \begin{corollary}\label{co-250118-1af}
  For  any  graph  $G$,  we  have  a  persistent  covering  map  
      \begin{eqnarray}\label{eq-250115-1}
     \pi(k,-):   {\rm  Conf}_k(V_G,-)\longrightarrow    {\rm  Conf}_k(V_G,-)/\Sigma_k
    \end{eqnarray}  
   and   persistent  vector  bundles  
    \begin{eqnarray}\label{eq-03.77}
 \boldsymbol{\xi}(V_G,k,-;\mathbb{F}) \cong    \boldsymbol{\zeta}(V_G,k,-;\mathbb{F})\oplus   \boldsymbol{\epsilon}(V_G,k,-;\mathbb{F}).  
\end{eqnarray}
 Here   the  persistent  vector  bundles  in  (\ref{eq-03.77})  
 are over  the  persistent  CW-complex 
  \begin{eqnarray*}
   {\rm  sk}^{k-1}  ({\rm  Ind}(G,-))\setminus  {\rm  sk}^{k-2}  ({\rm  Ind}(G,-)).
   \end{eqnarray*}  
  \end{corollary}
  \begin{proof}
    For  any  $n=1,2,\ldots$,   
         \begin{eqnarray*}
        {\rm  Conf}_{k} (V_G,\frac{n}{2} )/\Sigma_k = {\rm  sk}^{k-1}  ({\rm  Ind}(G,\frac{n}{2}))\setminus  {\rm  sk}^{k-2}.    ({\rm  Ind}(G,\frac{n}{2})).   
         \end{eqnarray*}  
         For  any  $n_1\leq  n_2$,   the  canonical  inclusion  
         \begin{eqnarray*}
         \iota_{n_2,n_1}:  {\rm  Ind}(G,\frac{n_2}{2})\longrightarrow  
         {\rm  Ind}(G,\frac{n_1}{2})
         \end{eqnarray*}
         induces  the  canonical  inclusion 
           \begin{eqnarray*}
         \iota_{n_2,n_1}:   {\rm  Conf}_{k} (V_G,\frac{n_2}{2} )/\Sigma_k\longrightarrow  
          {\rm  Conf}_{k} (V_G,\frac{n_1}{2} )/\Sigma_k.   
         \end{eqnarray*}
         Hence   
  \begin{eqnarray}\label{eq-250210-a}
  {\rm  Conf}_{k} (V_G,- )/\Sigma_k =  {\rm  sk}^{k-1}  ({\rm  Ind}(G,-))\setminus  {\rm  sk}^{k-2}  ({\rm  Ind}(G,-)) 
   \end{eqnarray}
   is  a      persistent  CW-complex.   
  Let  $X$  be  $V_G$  in  Lemma~\ref{le-2.aa1}.  We  obtain  the  persistent  covering  map  (\ref{eq-250115-1}).  
   Let  $X$  be  $V_G$  in  Proposition~\ref{pr-3.2aaz}.   
   We  obtain  the persistent  vector  bundles  and  the  persistent   
  Whitney  sum  (\ref{eq-03.77}). 
  \end{proof}

           The  next  corollary  follows  from  Theorem~\ref{th-25-5.1.1}
           and  Corollary~\ref{co-25-01-15-1}.    
 
 \begin{corollary}\label{pr-250116-9s}
    For  any  graph  $G$,  
    we  have  a  persistent  chain  complex  
 \begin{eqnarray}\label{eq-5.lllpq28}
(C_\bullet(  { \rm  Conf}_\bullet (V_G,-)), \boldsymbol{\partial})   
 \end{eqnarray}
with  a  $\Sigma_\bullet$-action   such  that  
 the  diagram   commutes 
\begin{eqnarray}\label{eq-250116-997a}
\xymatrix{
 H^*( { \rm  Conf}_k (V_G,\dfrac{n_2}{2}))\ar[d]  \ar[r]^-{\sigma^* }  
 & H^*( { \rm  Conf}_k (V_G,\dfrac{n_2}{2}))\ar[d]\\
  H^*( { \rm  Conf}_k (V_G,\dfrac{n_1}{2}))\ar[r]^-{\sigma^* }
  & H^*( { \rm  Conf}_k (V_G,\dfrac{n_1}{2}))
}
\end{eqnarray} 
for  any    $n_1\geq  n_2 $,  any  positive  integer  $k$  and  any  $\sigma\in \Sigma_k$. 
 \end{corollary}

     \begin{proof}
     It  follows  from   the  persistent  covering  map  
     \begin{eqnarray*}
     \pi(k,-): {\rm  Conf}_{k} (V_G,- )\longrightarrow    {\rm  Conf}_{k} (V_G,- )/\Sigma_k
     \end{eqnarray*}
   and   (\ref{eq-250210-a})  that  
     $ {\rm  Conf}_{k} (V_G,- )$  is  a  persistent  CW-complex.  
     Let  $X$  be  $V_G$  in  Theorem~\ref{th-25-5.1.1}.  
     Then    
     $\mathbf{d}$  is  trivial.  
     The  persistent   double  complex   (\ref{eq-5.lllq2})  
     as  well  as   the    persistent  chain  complex   (\ref{eq-5.lllq77}) 
    reduces to  the  persistent  chain  complex  (\ref{eq-5.lllpq28}).       
         Let  $X$  be  $V_G$  in  Corollary~\ref{co-25-01-15-1}.    
         The  commutative  diagram  (\ref{eq-250116-91})  gives  
         (\ref{eq-250116-997a}).  
              \end{proof}

\section{Obstructions for  strong  totally  geodesic  embeddings}\label{s5}

In  this  section,   we  apply  the  persistent   bundles  over  configuration  spaces  constructed 
in  Section~\ref{25-sect4}   to  give  obstructions  for   totally  geodesic  embeddings.  
In Subsection~\ref{ss.f.1},  we  define  the  strong totally  geodesic  embeddings
  of  Riemannian  manifolds   (cf.    
Definition~\ref{def-25-01-7})  and  give   obstructions   for  strong  totally  geodesic  embeddings
of  Riemannian  manifolds  
  in  Theorem~\ref{th-main1-25jan}.  
  In   Subsection~\ref{ss.f.2},  we  define  the  geodesic  embeddings  of  graphs 
  (cf.    Definition~\ref{def-250120-9})
    and   give   obstructions   for  strong  totally  geodesic  embeddings
of  graphs  
  in  Theorem~\ref{th-main2-25jan}.

\subsection{Obstructions for   strong  totally  geodesic  embeddings  of   
Riemannian  manifolds}\label{ss.f.1}

Let  $M'$  be  a  Riemannian  manifold.  
A  Riemannian  submanifold  $M\subseteq   M'$  is   called  {\it  totally  geodesic}  if  every 
geodesic  of  $M$  is  also  a  geodesic  of  $M'$   (cf.  \cite[p. 104,  Proposition~13]{oneill}).  
We  give  a    condition  stronger 
 than the totally  geodesic  condition  in  the  next  definition.

\begin{definition}\label{def-25-01-7}
We  call  a  Riemannian  submanifold  $M\subseteq   M'$  {\it  strong  totally  geodesic}  
if  for  any  $p,q\in   M$,    every  geodesic  of  $M'$  connecting  $p$ and  $q$ 
is  also  a  geodesic   of  $M$.  
\end{definition}

    \begin{example}
    Let  $m_1<m_2$  be  positive  integers.  
    The  unit  sphere   $S^{m_1}$  is  a  totally  geodesic  submanifold  
    of  the  unit sphere  $S^{m_2}$.  
    However,   $S^{m_1}$  is  not  a  strong  totally geodesic  submanifold  of  $S^{m_2}$  
    since  
   any  pair  of  antipodal  points  in $S^{m_1}$  
   can  be  connected  by  the semi-circles  of  any  great  circles
in  $S^{m_2}$.  
    \end{example}

\begin{lemma}\label{le-25-01-geod1a}
Let  $M'$  be  a  complete  Riemannian  manifold.  
 Then  
any   strong  totally  geodesic submanifold   $M\subseteq  M'$   is  totally  geodesic.    
\end{lemma}

\begin{proof}
Let  $M\subseteq   M'$   be  a  strong  totally  geodesic  submanifold.  
Let  $\gamma$  be  a  geodesic  of  $M$.  
It  suffices to  prove  $\gamma$  is  a  geodesic  of  $M'$.  
Otherwise,  by  the  completeness  of  $M'$  and  the  Hopf-Rinow Theorem
(cf.  \cite[p.  138]{oneill}),  
 we  can  choose   $p,q\in \gamma$  
 and  a  geodesic  $\gamma'$  of  $M'$  from  $p$  to  $q$  
such  that 
 \begin{eqnarray}\label{eq-250210-b}
 {\rm  length}(\gamma')<{\rm  length}(\gamma_{p,q}).  
 \end{eqnarray}
  Here   $\gamma_{p,q}$   is  the  shortest  geodesic  of  $M$  from  $p$  to  $q$,  
   i.e.   $\gamma_{p,q}$  is the  segment  of  $\gamma$  from  $p$  to  $q$.   
 Since  $M$  is  a  strong  totally  geodesic  submanifold  of  $M'$, 
 by  Definition~\ref{def-25-01-7},   
     $\gamma'$  is  a  geodesic  of  $M$  from  $p$  to  $q$.  
  This  contradicts   with  (\ref{eq-250210-b})
  and   that  $\gamma_{p,q}$   is  the  shortest  geodesic  of  $M$  from  $p$  to  $q$.    
\end{proof}

\begin{corollary}\label{co-25-01-geod1}
Let  $M'$  be  a  complete  Riemannian  manifold.  
 Then  a  Riemannian  submanifold  $M\subseteq   M'$  is   strong  totally  geodesic  iff
  for  any  $p,q\in   M$  and  any  curve  $\gamma$  in  $M'$  from  $p$  to  $q$,  
  \begin{eqnarray}\label{eq-2501geod1}
   \gamma  {\rm~ is~  a ~ geodesic~  of~  }M  \Longleftrightarrow   \gamma
   {\rm~  is ~ a ~ geodesic ~ of ~}  M'. 
  \end{eqnarray}
\end{corollary}

\begin{proof}
The  corollary  follows  from  Definition~\ref{def-25-01-7}  
and  Lemma~\ref{le-25-01-geod1a}. 
\end{proof}

\begin{lemma}
\label{le-250210-ay}
Let  $M'$  be  a    Riemannian  manifold.  
Let  $M $  be  a  Riemannian  
    submanifold  of  $M'$    
     that   does  not  contain  any   pair  of  conjugate  points  of  $M'$. 
     If  $M\subseteq  M'$  is  totally  geodesic,  then  
     $M\subseteq  M'$  is  strong  totally  geodesic.  
\end{lemma}

\begin{proof}
Suppose $M\subseteq  M' $  is   totally  geodesic.  
 Let  $p,q\in  M$.  
 For  any  geodesic     $\gamma_{p,q}$    of  $M$  from  $p$  to  $q$,     
    it  is  also a  geodesic  of  $M'$.  
 Since  $M$  does  not  contain  any   pair  of  conjugate  points  of  $M'$, 
  for  any  $a \in  \gamma_{p,q}$,  
  it  follows   that    $p$  and  $a$  are     not  conjugate  in  $M'$,     
  and   $a$  and  $q$  are  not  conjugate  in  $M'$. 
  Thus  
  for  any  geodesic     $\gamma_{p,q}$    of  $M$  from  $p$  to  $q$, 
   it  is  the  unique  geodesic 
 of  $M'$  from  $p$  to  $q$.  
 By  Definition~\ref{def-25-01-7},  
  $M\subseteq  M' $  is  strong  totally  geodesic.
\end{proof}

    \begin{lemma}
    \label{le-250118-abc1}
    Let  $M'$  be  a  Riemannian  manifold.  
    Let  $M $  be  a  Riemannian  
    submanifold  of  $M'$.  
    Then  
    \begin{eqnarray}\label{eq-1.77ab} 
    d_{M}(p,q)\geq  d_{M'}(p,q)
    \end{eqnarray}
    for  any  $p,q\in  M$.  
Moreover,   
\begin{enumerate}[(1)]
\item
if    $M\subseteq  M'$   is    strong    totally geodesic,    
 then   the  equality  of  (\ref{eq-1.77ab})  holds  for  any  $p,q\in  M$;     
  \item
  If  the  equality  of  (\ref{eq-1.77ab})  holds  for  any  $p,q\in  M$,
  then  $M\subseteq  M'$  is  totally  geodesic;    
  \item
  If  $M$  does  not  contain  any   pair  of  conjugate  points  of  $M'$   
  and  
   the  equality  of  (\ref{eq-1.77ab})  holds  for  any  $p,q\in  M$,  
  then     $M\subseteq  M'$   is    strong    totally geodesic.  
  \end{enumerate}
    \end{lemma}
    
    \begin{proof}
    Let  $M\subseteq  M'$  be  a  Riemannian  submanifold.  
    Then  for  any  $p,q\in  M$,  
    \begin{eqnarray*}
d_M(p,q)&=&
\inf_{\gamma:[0,1]\longrightarrow  M\atop \gamma(0)=p{\rm~and~}\gamma(1)=q} {\rm  length}(\gamma) \\
&\geq  &  \inf_{\gamma':[0,1]\longrightarrow  M'\atop \gamma'(0)=p{\rm~and~}\gamma'(1)=q} {\rm  length}(\gamma')\\
&=&   d_{M'}(p,q).  
\end{eqnarray*}

(1)  Suppose  $M\subseteq  M'$   is    strong    totally geodesic.   
Let    $p,q\in  M$.  
Then  by  Definition~\ref{def-25-01-7},       
any  geodesic  $\gamma'$  of  $M'$  in  from  $p$  to  $q$  must  be  
a  geodesic  $\gamma$  of  $M$.  
Thus  
\begin{eqnarray*}
d_{M'}(p,q)&=&
  \inf_{\gamma':[0,1]\longrightarrow  M'{\rm~is~a~geodesic~of~}M'\atop \gamma'(0)=p{\rm~and~}\gamma'(1)=q} {\rm  length}(\gamma')\\
&\geq & \inf_{\gamma:[0,1]\longrightarrow  M{\rm~is~a~geodesic~of~}M\atop \gamma(0)=p{\rm~and~}\gamma(1)=q} {\rm  length}(\gamma)\\
&=& d_{M}(p,q). 
\end{eqnarray*}
We  obtain the  equality  of  (\ref{eq-1.77ab}).

  (2)  Suppose  the  equality  of  (\ref{eq-1.77ab})  holds  for  any  $p,q\in  M$.  
   To  prove  (2),  we  suppose  to  the  contrary  that   $M\subseteq  M'$  is  not  totally  geodesic. 
   Then  we  can  choose  $p,q\in M$  and 
     a    geodesic  $\gamma_{p,q}$  of  $M$  from  $p$ to  $q$ 
   such  that  $\gamma_{p,q}$  is  not  a  geodesic  of  $M'$.  
   Without  loss  of  generality,  
   we  can choose  $p$  and  $q$  properly  such that  $\gamma_{p,q}$
   is     shortest  in  $M$.  
  It  follows  that   
  \begin{eqnarray*}
  d_{M'}(p,q)< {\rm  length}(\gamma_{p,q})= d_M(p,q), 
  \end{eqnarray*}
   which  contradicts  with  our  assumption. 
   Thus  $M\subseteq  M'$  is    totally  geodesic.

  (3)  Suppose  for  any  $p,q\in  M$,   $p$  and  $q$  are  not  conjugate  in  $M'$  and  
the  equality  of  (\ref{eq-1.77ab})  is  satisfied.  
By   (2),   $M\subseteq  M'$  is  totally  geodesic.  By   Lemma~\ref{le-250210-ay},  
  $M\subseteq  M'$  is  strong   totally  geodesic.    
    \end{proof}
    
    \begin{corollary}\label{co-geod2}
    Let 
    $M'$  be  a  complete  Riemannian  manifold.   
   Let  $M $  be  a  Riemannian  
    submanifold  of  $M'$    
     that   does  not  contain  any   pair  of  conjugate  points  of  $M'$. 
    Then   the  followings  are  equivalent 
    \begin{enumerate}[(1)]
    \item
    $M\subseteq  M'$   is   totally  geodesic;  
    \item
      $M\subseteq  M'$   is  strong    totally  geodesic; 
    \item
    $d_M(p,q)=d_{M'}(p,q)$    for  any  $p,q\in   M$.  
    \end{enumerate}
    \end{corollary}
    
    \begin{proof} 
    (2)  $\Longrightarrow$  (1):  Lemma~\ref{le-25-01-geod1a}.  
    (1)   $\Longrightarrow$  (2):  Lemma~\ref{le-250210-ay}.     
     (2)  $\Longrightarrow$  (3):  Lemma~\ref{le-250118-abc1}~(1).    
     (3)  $\Longrightarrow$  (2):  Lemma~\ref{le-250118-abc1}~(3).   
        \end{proof}

    The  next  lemma  follows  from  Lemma~\ref{le-250118-abc1}~(1).

  \begin{lemma}\label{le-mfd-3.1}
  Let  $\varphi: M\longrightarrow  M'$  
  be  a   strong  totally geodesic  embedding  of   Riemannian  manifolds.  
  Then  we  have  
  an  induced   pull-back  of  persistent  covering  maps  
    \begin{eqnarray}\label{eq-3.x2}
    \xymatrix{
    {\rm  Conf}_k(M,-)\ar[d]_-{\pi(k,-)} \ar[rr] ^-{ {\rm  Conf}(\varphi,-)}
    &&{\rm  Conf}_k(M',-)\ar[d]^-{\pi(k,-)} \\
     {\rm  Conf}_k(M, -)/\Sigma_k  \ar[rr] ^-{ {\rm  Conf}(\varphi,-)/\Sigma_k }
    &&{\rm  Conf}_k(M',-)/\Sigma_k   
    }
  \end{eqnarray}
  for  any   positive  integer  $k$.  
  \end{lemma}
  
  \begin{proof}
By  Lemma~\ref{le-250118-abc1}~(1),     for  any  positive  integer  $k$  and  any  
$r\geq  0$,  $\varphi$  induces  a   $\Sigma_k$-equivariant  embedding   
\begin{eqnarray*}
{\rm  Conf}_k(\varphi,r):  {\rm  Conf}_k(M,r)\longrightarrow  {\rm  Conf}_k(M',r)
\end{eqnarray*}
such  that  the  diagrams  commute 
\begin{eqnarray*}
\xymatrix{
{\rm  Conf}_k(M,r_2)\ar[r] \ar[d] _-{{\rm  Conf}_k(\varphi,r_2)} 
&{\rm  Conf}_k(M,r_1)\ar[d]^-{{\rm  Conf}_k(\varphi,r_1)}\\
{\rm  Conf}_k(M',r_2) \ar[r] &{\rm  Conf}_k(M',r_1), 
}
\end{eqnarray*}
  where  the  horizontal  maps  are  canonical  inclusions,   
for  any  $r_2\geq  r_1\geq  0$.  
Hence  we  have  a  filtered  $\Sigma_k$-equivariant  map  
\begin{eqnarray*}
{\rm  Conf}_k(\varphi,-):  {\rm  Conf}_k(M,-)\longrightarrow  {\rm  Conf}_k(M',-)  
\end{eqnarray*}
which  induces  a  pull-back of  persistent  covering  maps  by  the  commutative  diagram  (\ref{eq-3.x2}).    
  \end{proof}

    \begin{theorem}
    \label{th-main1-25jan}
      Let  $\varphi: M\longrightarrow  M'$  
  be  a   strong  totally geodesic  embedding  of    Riemannian  manifolds.  
  Then  we  have   induced  a  pull-back    of  persistent  vector  bundles  
  \begin{eqnarray}\label{eq-250120-1}
   \boldsymbol{\xi}(M,k,-;\mathbb{F}) = \big({\rm  Conf}(\varphi,-)/\Sigma_k\big)^*
      \boldsymbol{\xi}(M',k,-;\mathbb{F}). 
  \end{eqnarray}
  Consequently,  we  have  the  pull-back  persistent  Stiefel-Whitney  class
    \begin{eqnarray}\label{eq-250120-7}
  w( \boldsymbol{\xi}(M,k,-;\mathbb{R})) = \big({\rm  Conf}(\varphi,-)/\Sigma_k\big)^*
      w(\boldsymbol{\xi}(M',k,-;\mathbb{R})) 
  \end{eqnarray}
  and  the   pull-back   persistent  Chern  class 
     \begin{eqnarray}\label{eq-250120-8}
  c( \boldsymbol{\xi}(M,k,-;\mathbb{C})) = \big({\rm  Conf}(\varphi,-)/\Sigma_k\big)^*
      c(\boldsymbol{\xi}(M',k,-;\mathbb{C})).
  \end{eqnarray}
    \end{theorem}
    
    \begin{proof}
    By  Lemma~\ref{le-mfd-3.1},   the  pull-back  (\ref{eq-3.x2})   of  persistent  covering  maps    
    induces  a  pull-back (\ref{eq-250120-1}) of  the  associated persistent  vector  bundles.

    Let  $\mathbb{F}=\mathbb{R}$  in   (\ref{eq-250120-1}).  
    Take  the  persistent  cohomology  ring  with  coefficients  in  $\mathbb{Z}_2$.  
     With  the  help  of   Definition~\ref{def-250120-1},  
     we  obtain (\ref{eq-250120-7}).

     Let  $\mathbb{F}=\mathbb{C}$  in   (\ref{eq-250120-1}).  
    Take  the  persistent  cohomology  ring  with  coefficients  in  $\mathbb{Z}$.  
     With  the  help  of   Definition~\ref{def-250120-2},  
     we  obtain (\ref{eq-250120-8}).  
    \end{proof}

\subsection{Obstructions for  strong  totally geodesic  embeddings  of   graphs}
\label{ss.f.2}

A  graph  $G=(V_{G},  E_{G})$     is  a  {\it  subgraph} 
of    $G'$    if  
 $V_{G}\subseteq  V_{G'}$  and 
  $E_{G}\subseteq E_{G'}$. 
  Let  $G $  be  a  subgraph  of  $G' $.  
  Then  for  any  $u,v\in  V_{G}$,  by  an  analog  of  Lemma~\ref{le-250118-abc1},  
\begin{eqnarray}\label{eq-1.9}
d_{G}(u,v)\geq  d_{G'}(u,v). 
\end{eqnarray}
The  next  definition  is  inspired  by  Corollary~\ref{co-geod2}.  

\begin{definition}\label{def-250120-9}
We  say  that  
$G$  is  a  {\it   strong   totally  geodesic  subgraph}
  of   $G'$  if  the  equality  in  (\ref{eq-1.9})   is  satisfied  for  any  $u,v\in  V_{G}$.  
  \end{definition}
  
  Let  $G$  and  $G'$  be  two graphs.  
  A {\it  morphism}  $\varphi:  G\longrightarrow  G'$  
  is  a  map  $\varphi: V_G\longrightarrow  V_{G'}$  
  such  that   either  $\varphi(u)=\varphi(v)$  or   
   $\{\varphi(u),\varphi(v)\}\in  E_{G'}$   for  any  $\{u,v\}\in  E_G$.  
  Let  $\varphi:  G\longrightarrow  G'$  
  be  a  morphism.  
  We  say  that 
  $\varphi$  is  a  {\it  strong  totally  geodesic  embedding}  
  if  $\varphi: V_G\longrightarrow  V_{G'}$  is  injective  
  and  the  image  $\varphi(G)= (\varphi(V_G), \varphi(E_G))$  
  is  a  strong  totally  geodesic  subgraph  of  $G'$.  
    In this case,  if  we  identify  $G$  and  $\varphi(G)$,  
    then  $G$  is  a      strong   totally  geodesic  subgraph 
  of   $G'$.

  \begin{lemma}\label{le-mfd-3.1y}
  Let  $\varphi: G\longrightarrow  G'$  
  be  a   strong  totally geodesic  embedding  of  graphs.  
  Then  we  have  
  an  induced   pull-back  of  persistent  covering  maps  
    \begin{eqnarray}\label{eq-3.y2}
    \xymatrix{
    {\rm  Conf}_k(V_G,-)\ar[d]_-{\pi(k,-)} \ar[rr] ^-{ {\rm  Conf}(\varphi,-)}
    &&{\rm  Conf}_k(V_{G'},-)\ar[d]^-{\pi(k,-)} \\
     {\rm  Conf}_k(V_G, -)/\Sigma_k  \ar[rr] ^-{ {\rm  Conf}(\varphi,-)/\Sigma_k }
    &&{\rm  Conf}_k(V_{G'},-)/\Sigma_k   
    }
  \end{eqnarray}
  for  any   positive  integer  $k$.  
  \end{lemma}
  
  \begin{proof}
By  Definition~\ref{def-250120-9},   for  any  positive  integer  $k$  and  any  
$n\in \mathbb{N}$,  $\varphi$  induces  a   $\Sigma_k$-equivariant  embedding   
\begin{eqnarray*}
{\rm  Conf}_k(\varphi,\frac{n}{2}):  {\rm  Conf}_k(V_G,\frac{n}{2})\longrightarrow  {\rm  Conf}_k(V_{G'},\frac{n}{2})
\end{eqnarray*}
such  that  the  diagrams  commute 
\begin{eqnarray*}
\xymatrix{
{\rm  Conf}_k(V_G,\dfrac{n_2}{2})\ar[r] \ar[d] _-{{\rm  Conf}_k(\varphi,\frac{n_2}{2})} 
&{\rm  Conf}_k(V_G,\dfrac{n_1}{2})\ar[d]^-{{\rm  Conf}_k(\varphi,\frac{n_1}{2})}\\
{\rm  Conf}_k(V_{G'},\dfrac{n_2}{2}) \ar[r] &{\rm  Conf}_k(V_{G'},\dfrac{n_1}{2})
}
\end{eqnarray*}
for  any  nonnegative  integers  $n_2\geq  n_1$  where  the  horizontal  maps  are  canonical  inclusions.  
  Analogous with  Lemma~\ref{le-mfd-3.1},  
  we  have a  commutative  diagram   (\ref{eq-3.y2})  
  which  is  a   pull-back  of  persistent  covering  maps.  
  \end{proof}
  
    \begin{theorem}
    \label{th-main2-25jan}
      Let  $\varphi:  G\longrightarrow  G'$  
  be  a   strong  totally geodesic  embedding  of  graphs.  
  Then  we  have   induced  a  pull-back    of  persistent  vector  bundles  
  \begin{eqnarray}\label{eq-250121-1}
   \boldsymbol{\xi}(V_G,k,-;\mathbb{F}) = \big({\rm  Conf}(\varphi,-)/\Sigma_k\big)^*
      \boldsymbol{\xi}(V_{G'},k,-;\mathbb{F}). 
  \end{eqnarray}
  Consequently,  we  have  the  pull-back  persistent  Stiefel-Whitney  class
    \begin{eqnarray}\label{eq-250121-7}
  w( \boldsymbol{\xi}(V_G,k,-;\mathbb{R})) = \big({\rm  Conf}(\varphi,-)/\Sigma_k\big)^*
      w(\boldsymbol{\xi}(V_{G'},k,-;\mathbb{R})) 
  \end{eqnarray}
  and  the   pull-back   persistent  Chern  class 
     \begin{eqnarray}\label{eq-250121-8}
  c( \boldsymbol{\xi}(V_G,k,-;\mathbb{C})) = \big({\rm  Conf}(\varphi,-)/\Sigma_k\big)^*
      c(\boldsymbol{\xi}(V_{G'},k,-;\mathbb{C})).
  \end{eqnarray}
    \end{theorem}
    
    \begin{proof}
    By  Lemma~\ref{le-mfd-3.1},   the  pull-back  (\ref{eq-3.y2})   of  persistent  covering  maps    
    induces  a  pull-back (\ref{eq-250121-1}) of  the  associated persistent  vector  bundles.

    Let  $\mathbb{F}=\mathbb{R}$  in   (\ref{eq-250121-1}).  
    Take  the  persistent  cohomology  ring  with  coefficients  in  $\mathbb{Z}_2$.  
     With  the  help  of   Definition~\ref{def-250120-1}  and  Corollary~\ref{co-250118-1af},  
     we  obtain (\ref{eq-250121-7}).

     Let  $\mathbb{F}=\mathbb{C}$  in   (\ref{eq-250121-1}).  
    Take  the  persistent  cohomology  ring  with  coefficients  in  $\mathbb{Z}$.  
     With  the  help  of   Definition~\ref{def-250120-2}  and  Corollary~\ref{co-250118-1af},  
     we  obtain (\ref{eq-250121-8}).  
    \end{proof}

\section{Obstructions  for  regular  embeddings}\label{s7}

In  this  section,   we  apply  the  persistent   bundles  over  configuration  spaces  constructed 
in  Section~\ref{25-sect4}   to  give  obstructions  for   regular    embeddings
in   real  and  complex   Euclidean  spaces.  
We  use  the  Stiefel-Whitney  class  of  the  persistent  bundle to  give  obstructions  
for  regular  embeddings  in  $\mathbb{R}^N$  in  Theorem~\ref{25-cor1} 
and  use  the  Chern  class  of  the  persistent  bundle to  give  obstructions  
for  regular  embeddings  in  $\mathbb{C}^N$  in  Theorem~\ref{25-cor2}.   
With  the  help  of  Section~\ref{s5},  we  apply
  Theorem~\ref{25-cor1}   and  Theorem~\ref{25-cor2} 
to  give  obstructions  for  regular  embeddings  of  complete  Riemannian  manifolds  
in  Theorem~\ref{25-cor1zzz}  in Subsection~\ref{ss7.1} 
and  give  obstructions  for  regular  embeddings  of  graphs   
in Theorem~\ref{25-cor2zzz}  and  Theorem~\ref{co-25-01-graph}  in  Subsection~\ref{ss7.2}.

Let  $(X,d)$  be  a  metric  space  where 
$X$  is  a  CW-complex  and   $d:  X\times  X\longrightarrow  [0,+\infty]$  is  
a  metric  on    $X$.  
The  next  definition  is  a  generalization  of  the  regular  embeddings  
in  \cite{high1,high2, Borsuk, chi, cohen1, handel1,handel2, 1996,handel3, topapp}.  

\begin{definition}\label{def-7.1}
We  call  a  map  $f:  X\longrightarrow  \mathbb{F}^N$ 
to  be  
{\it   $(k,r)$-regular}  if  
for  any  distinct  $k$-points  $x_1, \ldots, x_k\in  X$ 
such  that  $d(x_i,x_j)>2r$  for  any  $i\neq  j$,  
their  images  
$f(x_1)$,  $\ldots$,  $f(x_k)$ are  linearly  independent  over 
$\mathbb{F}$.  
We  call  a   $(k,0)$-regular  map    $f:  V_G\longrightarrow  \mathbb{F}^N$  simply  a  {\it  $k$-regular  map}.  
\end{definition}

The  following  observations  are  direct.   
\begin{enumerate}[(1)]
\item
A  $(k,r)$-regular  map  $f:  X\longrightarrow  \mathbb{F}^N$
 is  $(k',r')$-regular  for  
any  $k'\leq  k$  and  any  $r'\geq  r$;  

\item
A   $(k,r)$-regular  map   $f:  X\longrightarrow  \mathbb{F}^N$
is  injective  for   $k\geq 2$;
\item
A   $k$-regular  map    $f:  X\longrightarrow  \mathbb{F}^N$ 
sends  any    $k$   distinct   points     in   $X$
to  a   linearly  independent   $k$  vectors   in   $\mathbb{F}^N$;  

\item
Let  $Y\subseteq  X$.  
  If  $f:  X\longrightarrow  \mathbb{F}^N$ is  a
  $k$-regular  map,  
  then   the  restriction  of  $f$  induces  a  $k$-regular  map  
  $f:  Y\longrightarrow  \mathbb{F}^N$; 
  
  \item
  Let  $Y\subseteq  X$  such  that  
  $d_Y(x,y)=d_X(x,y)$  for  any  $x,y\in  Y$.  
   If  $f:  X\longrightarrow  \mathbb{F}^N$ is  a
  $(k,r)$-regular  map,  
  then   the  restriction  of  $f$  induces  a  $(k,r)$-regular  map  
  $f:  Y\longrightarrow  \mathbb{F}^N$. 
 \end{enumerate}

  Consider  the  persistent  CW-complex  
  \begin{eqnarray*}
  {\rm  Conf}_k(X,-+r)=\{{\rm  Conf}_k(X,r'+r)\mid  r'\geq  0\}   
  \end{eqnarray*}
  with  the   canonical  persistent  $\Sigma_k$-action.  
  We  have  an  orbit  persistent  CW-complex   
  \begin{eqnarray*}
   {\rm  Conf}_k(X,-+r)/\Sigma_k=\{{\rm  Conf}_k(X,r'+r)/\Sigma_k\mid  r'\geq  0\}.   
   \end{eqnarray*}     
  We  have   a   persistent  vector  bundle  
  \begin{eqnarray*}
  \boldsymbol{\xi}(X,k,-+r;\mathbb{F})=\{ {\xi}(X,k,r'+r;\mathbb{F})\mid  r'\geq  0\}
  \end{eqnarray*}
  over      $  {\rm  Conf}_k(X,-+r)/\Sigma_k$.

  Let  $G_k(\mathbb{F}^N)$  and    $G_k(\mathbb{F}^\infty)$  be  the  Grassmannian  
  manifolds  consisting   of  all  the  $k$-dimensional subspaces  of  $\mathbb{F}^N$
  and  all  the  $k$-dimensional subspaces  of  $\mathbb{F}^\infty$  respectively. 
    The  next  lemma   is    a  generalization  of  \cite[Proposition~4.1]{topapp}.

\begin{lemma}
Let  $f:  X\longrightarrow  \mathbb{F}^N$  be  a   $(k,r)$-regular  map.  
   Then  for  each  positive  integer  $k$,  we  have  an  induced   persistent   map  
   \begin{eqnarray*}
   f:  {\rm  Conf}_k(X,-+r)/\Sigma_k\longrightarrow   G_k(\mathbb{F}^N)
   \end{eqnarray*}
  such  that  the  composition 
   \begin{eqnarray}\label{eq-250122-1}
   {\rm  Conf}_k(X,-+r)/\Sigma_k\overset{f}{\longrightarrow}   G_k(\mathbb{F}^N)\overset{\iota}{\longrightarrow}
    G_k(\mathbb{F}^\infty)
   \end{eqnarray}
   is   the  classifying  map  of  the  persistent  vector  bundle
      $\boldsymbol{\xi}(M,k,-+r;\mathbb{F})$.  
\end{lemma}

 \begin{proof}  
  Let  $r'\geq  0$.     Let  $V_k(\mathbb{F}^N)$  and    $V_k(\mathbb{F}^\infty)$  be  the  Stiefel   
  manifolds  consisting   of  all  the  $k$-frames   in   $\mathbb{F}^N$
  and   all  the  $k$-frames  in  $\mathbb{F}^\infty$  respectively.  
 We  have  a  commutative  diagram 
    \begin{eqnarray*}
  \xymatrix{
         {\rm  Conf}_k(X,r'+r) \ar[d]_{\pi_k}\ar[r]  & V_k(\mathbb{F}^N)  
      \ar[d]\ar[r]  &V_k(\mathbb{F}^\infty)\ar[d] \\
      {\rm  Conf}_k(X,r'+r)/\Sigma_k\ar[r]^-{f} &  G_k(\mathbb{F}^N)\ar[r]^-{\iota} &
    G_k(\mathbb{F}^\infty)
  }
\end{eqnarray*}  
where   the  first  vertical  map   is  the   projection   modulo  the  $\Sigma_k$-action  on  
the  coordinates  and  the  second  and  the  third  vertical  maps  are  the  projections  
sending   a  $k$-frame    to  the  vector  space  spanned  by  the  $k$-frame.  
Let $E_k(\mathbb{F}^N)$  be the manifold consisting of 
 $(l,v)\in G_k(\mathbb{F}^N)\times \mathbb{F}^k$
  such that $l$ is a $k$-dimensional subspace of $\mathbb{F}^N$ and $v\in l$.  
Let $E_k(\mathbb{F}^\infty)$ be the manifold consisting of
 $(l',v')\in G_k(\mathbb{F}^\infty)\times \mathbb{F}^k$ 
 such that $l$ is a $k$-dimensional subspace of 
 $\mathbb{F}^\infty$ and $v'\in l'$. 
 Then the following diagram commutes and each square is a pull-back
\begin{eqnarray*}
\xymatrix{
 {\rm  Conf}_k(X,r'+r)\times_{\Sigma_k}\mathbb{F}^k\ar[r]\ar[d] &E_k(\mathbb{F}^N)\ar[r]\ar[d]  & E_k(\mathbb{F}^\infty)\ar[d] \\
 {\rm  Conf}_k(X,r'+r)/\Sigma_k\ar[r]^-f &G_k(\mathbb{F}^N)\ar[r]^-\iota&G_k(\mathbb{F}^\infty). }  
\end{eqnarray*}
Thus $\iota\circ  f$ is a pull-back of vector bundles 
and  is     a classifying map of  $\xi(X,k,r'+r;\mathbb{F})$   as a $O(\mathbb{F}^k)$-bundle.  
Let  $r'_1\geq  r'_2\geq  0$.   The  diagram  commutes
\begin{eqnarray*}
\xymatrix{
 {\rm  Conf}_k(X,r'_1+r)/{\Sigma_k}\ar[r]^-f \ar[d] &G_k(\mathbb{F}^N) \ar@{=}[d] \\
 {\rm  Conf}_k(X,r'_2+r)/\Sigma_k\ar[r]^-f &G_k(\mathbb{F}^N) 
  }  
\end{eqnarray*}
where  the    vertical  arrow  on  the  left   is   the  canonical  inclusion.  
Therefore,  $\iota\circ  f$ is a pull-back of  persistent  vector bundles 
and  is     a classifying map of  the  persistent  vector  bundle 
 $\boldsymbol{\xi}(X,k,-+r;\mathbb{F})$. 
  \end{proof}
  
    Let $R$ be a commutative  ring  with  unit. 
 Apply   the   functor  of   persistent   cohomology  rings   to  (\ref{eq-250122-1}).  
 We  obtain   induced  persistent  homomorphisms  of  persistent  cohomology  rings 
 such  that  the     diagram  commutes  
\begin{eqnarray}
\xymatrix{
H^*(G_k(\mathbb{F}^\infty);R)\ar[r]^-{\iota^*} \ar[rd]_-{(\iota\circ  f)^*} &H^*(G_k(\mathbb{F}^N);R) 
\ar[d]^-{f^*}  \\
&H^*({\rm  Conf}_k(X,-+r)/\Sigma_k;R).     
}
\label{eq-25-diag}
\end{eqnarray}
The  next   theorem   follows  from  (\ref{eq-25-diag})  with  $\mathbb{F}=\mathbb{R}$.

\begin{theorem} \label{25-cor1} 
Let
 $f: X\longrightarrow \mathbb{R}^N$ be  
 a $(k,r)$-regular map.  
 If  
 \begin{eqnarray}\label{eq-25-901}
 \bar w_{t(k)}( \boldsymbol{\xi}(X,k,-+r;\mathbb{R}))\neq 0 
 \end{eqnarray}
   for  some  $t(k)\geq  1$,  then 
 \begin{eqnarray}\label{eq-3.1x}
 N\geq   t(k)+k.
 \end{eqnarray}
\end{theorem}
\begin{proof}
Suppose $f:X\longrightarrow \mathbb{R}^N$ is a  $(k,r)$-regular map. 
Let  $R=\mathbb{Z}_2$  in  (\ref{eq-25-diag}).  
 Then     
  the  induced  map  
  \begin{eqnarray*}
  \iota^*:  H^*(G_k(\mathbb{R}^\infty);\mathbb{Z}_2)\longrightarrow 
  H^*(G_k(\mathbb{R}^{N});\mathbb{Z}_2)  
  \end{eqnarray*}  
  is     
  an  epimorphism  
  \begin{eqnarray}\label{eq-25-7.n}
  \iota^*:  \mathbb{Z}_2[w_1,w_2,\cdots,w_k] \longrightarrow  \frac{{\mathbb{Z}_2[w_1,w_2,\cdots,w_k]}}{ {(\bar w_{N-k+1},\bar w_{N-k+2},\cdots,\bar w_{N})}},   
  \end{eqnarray}
  where $w_i$ is the $i$-th universal Stiefel-Whitney class with $|w_i|=i$,   
 $\bar w_j$ is   the $j$-th degree term in the expansion of $(1+w_1+\cdots+w_k)^{-1}$ and 
 \begin{eqnarray*}
 (\bar w_{N-k+1},\bar w_{N-k+2},\cdots,\bar w_{N}) 
 \end{eqnarray*}
  is the ideal generated by $\bar w_{N-k+1}$, $\bar w_{N-k+2}$, $\cdots$, $\bar w_{N}$.

By  substituting  (\ref{eq-25-7.n})  in  (\ref{eq-25-diag}),  we  have  
   homomorphisms  of  persistent  cohomology  algebras 
    such  that  the     diagram  commutes  
\begin{eqnarray*}
\xymatrix{
  \mathbb{Z}_2[w_1,w_2,\cdots,w_k]\ar[r]^-{\iota^*} \ar[rd] _-{(\iota \circ  f)^*}
  & {\mathbb{Z}_2[w_1,w_2,\cdots,w_k]}/ {(\bar w_{N-k+1},\bar w_{N-k+2},\cdots,\bar w_{N})}
  \ar[d]^{f^*}\\
  & H^*({\rm  Conf}_k(X,-+r)/\Sigma_k;\mathbb{Z}_2). 
}
\end{eqnarray*}
It  follows   that 
\begin{eqnarray}\label{eq-250211-v1}
 (\iota \circ  f)^*\bar w_{N-k+1}  = 
 (\iota \circ  f)^*\bar w_{N-k+2}  
 =\cdots =(\iota \circ  f)^*\bar w_{N} =0
 \end{eqnarray}
    and  
   \begin{eqnarray}\label{eq-250211-v2}
  (\iota \circ  f)^*\bar w_{t(k)}=  \bar w_{t(k)}( \boldsymbol{\xi}(X,k,-+r;\mathbb{R})) \neq  0.  
 \end{eqnarray} 
   It  follows  from  (\ref{eq-250211-v1})  and  (\ref{eq-250211-v2})   that   
   either $t(k)\leq N-k $ or $t(k)\geq N+1$. Suppose $t(k)\geq N+1$. For any $M\geq N$, there exists a $(k,r)$-regular map $i\circ f$ from $X$ into $\mathbb{R}^M$
where $i$ is a  linear  embedding    of $\mathbb{R}^N$ into $\mathbb{R}^M$.
Hence by  applying  the  above  argument,  
either $t(k)\leq  M-k $ or $t(k)\geq  M+1$. Let $M=N+1,N+2,\cdots$ subsequently.
 We obtain $t(k)\geq N+2, N+3, \cdots$  which contradicts that $t(k)$ is finite. 
 Therefore, we  have   (\ref{eq-3.1x}).  
\end{proof}

The  next  corollary  follows  from  Theorem~\ref{25-cor1}. 

\begin{corollary}\label{25-cor1a} 
 Let
 $f: X\longrightarrow \mathbb{R}^N$ be  
 a $(k,r)$-regular map.  
 Then  
 \begin{eqnarray}\label{eq-3.1xxx}
 N\geq    \sup_{r'\geq  r} t(r',k)+k 
 \end{eqnarray}
 where  for  each  $r'\geq  r$,  
 \begin{eqnarray}\label{eq-250211-im-w}
 t(r',k)= \sup  \{t\mid  \bar w_{t}( {\xi}(X,k,r'+r;\mathbb{R}))\neq 0  \}.  
 \end{eqnarray}
\end{corollary}

\begin{proof}
 To  get  rid  of  the  persistence  in  Theorem~\ref{25-cor1},   
 we  note  that    (\ref{eq-25-901})   is    equivalent  to  the  condition  that  there  exists
   $r'\geq  0$  such  that  
     \begin{eqnarray*}
      \bar w_{t(k)}(  {\xi}(X,k,r'+r;\mathbb{R}))\neq 0.  
      \end{eqnarray*}  
Consequently, by Theorem~\ref{25-cor1},  we  obtain  (\ref{eq-3.1xxx}).  
\end{proof}

The  next   theorem   follows  from  (\ref{eq-25-diag})  with  $\mathbb{F}=\mathbb{C}$.

\begin{theorem} \label{25-cor2} 
Let
 $f: X\longrightarrow \mathbb{C}^N$ be  
 a  complex   $(k,r)$-regular map.  
 If  
 \begin{eqnarray}\label{eq-25-902}
 \bar  c_{t(k)}( \boldsymbol{\xi}(X,k,-+r;\mathbb{C}))\neq 0 
 \end{eqnarray}
   for  some  $t(k)\geq  1$,  then 
 \begin{eqnarray}\label{eq-3.1y}
 N\geq   t(k)+k.
 \end{eqnarray}
\end{theorem}
\begin{proof} 
Suppose $f:X\longrightarrow \mathbb{C}^N$ is a  complex  $(k,r)$-regular map. 
Let  $R=\mathbb{Z}$  in  (\ref{eq-25-diag}).  
 Then     
  the  induced  map  
  \begin{eqnarray*}
  \iota^*:  H^*(G_k(\mathbb{C}^\infty);\mathbb{Z})\longrightarrow 
  H^*(G_k(\mathbb{C}^{N});\mathbb{Z})  
  \end{eqnarray*}  
  is     
  an  epimorphism  
  \begin{eqnarray}\label{eq-25-7.nc}
  \iota^*:  \mathbb{Z}[c_1,c_2,\cdots,c_k] \longrightarrow  \frac{{\mathbb{Z} [c_1,c_2,\cdots,c_k]}}{ {(\bar  c_{N-k+1},\bar  c_{N-k+2},\cdots,\bar  c_{N})}},   
  \end{eqnarray}
  where $c_i$ is the $i$-th universal Chern  class with $|c_i|=2i$,   
 $\bar  c_j$ is   the $j$-th degree term in the expansion of $(1+c_1+\cdots+c_k)^{-1}$ and 
 \begin{eqnarray*}
 (\bar   c_{N-k+1},\bar  c_{N-k+2},\cdots,\bar  c_{N}) 
 \end{eqnarray*}
  is the ideal generated by $\bar  c_{N-k+1}$, $\bar  c_{N-k+2}$, $\cdots$, $\bar  c_{N}$.

By  substituting  (\ref{eq-25-7.nc})  in  (\ref{eq-25-diag}),  we  have  
   homomorphisms  of  persistent  cohomology  rings  
    such  that  the     diagram  commutes  
\begin{eqnarray*}
\xymatrix{
  \mathbb{Z}[c_1,  c_2,\cdots,c_k]\ar[r]^-{\iota^*} \ar[rd] _-{(\iota \circ  f)^*}
  & {\mathbb{Z} [c_1,c_2,\cdots,c_k]}/ {(\bar c_{N-k+1},\bar  c_{N-k+2},\cdots,\bar  c_{N})}
  \ar[d]^{f^*}\\
  & H^*({\rm  Conf}_k(X,-+r)/\Sigma_k;\mathbb{Z}).  
}
\end{eqnarray*}
It  follows   that 
\begin{eqnarray}\label{eq-250211-c1}
(\iota \circ  f)^*\bar  c_{N-k+1}= 
(\iota \circ  f)^*\bar  c_{N-k+2}= 
\cdots=(\iota \circ  f)^*\bar  c_{N}=0
\end{eqnarray}
   and  
   \begin{eqnarray}\label{eq-250211-c2}
  (\iota \circ  f)^*\bar c_{t(k)}=  \bar  c_{t(k)}( \boldsymbol{\xi}(X,k,-+r;\mathbb{C})) \neq  0.  
 \end{eqnarray} 
   By  an  analog  of  the  proof  of   Theorem~\ref{25-cor1},   
     (\ref{eq-3.1y})  follows  from  (\ref{eq-250211-c1})  and  (\ref{eq-250211-c2}).  
\end{proof}

The  next  corollary  follows  from  Theorem~\ref{25-cor2}. 

\begin{corollary}\label{25-cor2a} 
 Let
 $f: X\longrightarrow \mathbb{C}^N$ be  
 a   complex   $(k,r)$-regular map.  
 Then  
 \begin{eqnarray}\label{eq-3.1yyy}
 N\geq    \sup_{r'\geq  r} t(r',k)+k 
 \end{eqnarray}
 where  for  each  $r'\geq  r$,  
 \begin{eqnarray}\label{eq-250211-im-c}
 t(r',k)= \sup  \{t\mid  \bar  c_{t}( {\xi}(X,k,r'+r;\mathbb{C}))\neq 0  \}.  
 \end{eqnarray}
\end{corollary}

\begin{proof}
To  get  rid  of  the  persistence  in  Theorem~\ref{25-cor2},
  we   note  that    (\ref{eq-25-902})   is    equivalent  to  the  condition  that  there  exists
   $r'\geq  0$  such  that  
     \begin{eqnarray*}
      \bar  c_{t(k)}(  {\xi}(X,k,r'+r;\mathbb{C}))\neq 0.  
      \end{eqnarray*}  
Consequently, by Theorem~\ref{25-cor2},  we  obtain  (\ref{eq-3.1yyy}).  
\end{proof}

\subsection{Obstructions  for  regular  embeddings  of  Riemannian  manifolds}\label{ss7.1}

\begin{lemma}\label{le-25jan09z}
 Let  $M'$  be  a     Riemannian  manifold. 
    Let  $M$  be  a   strong  totally  geodesic  submanifold
  of   $M'$.    
   If  $f:  M'\longrightarrow  \mathbb{F}^N$ is  a
  $(k,r)$-regular  map,  
  then   the  restriction  of  $f$  to  $M$   induces  a  $(k,r)$-regular  map  
  $f:  M\longrightarrow  \mathbb{F}^N$.    
  \end{lemma}

\begin{proof}
The  lemma  follows  from  Lemma~\ref{le-250118-abc1}~(1)  and   the  observation  (5)  
  of   Definition~\ref{def-7.1}.  
\end{proof}

Theorem~\ref{25-cor1} and  Theorem~\ref{25-cor2}  can  be  applied  to  
    (real  and   complex)  $(k,r)$-regular   embedding  problems  of  Riemannian  manifolds 
where  the  distance   is  given  by  (\ref{eq-dist-riem}).

\begin{theorem} \label{25-cor1zzz} 
Let  $M'$  be  a    Riemannian  manifold.  
\begin{enumerate}[(1)]
\item

Let
 $f: M'\longrightarrow \mathbb{R}^N$ be  
 a $(k,r)$-regular map.  
 Then 
  \begin{eqnarray}\label{eq-3.1zzz}
 N\geq   \sup\{t(k)\mid  \bar w_{t(k)}( \boldsymbol{\xi}(M,k,-+r;\mathbb{R}))\neq 0 \}+k 
 \end{eqnarray}
 where  $M$  runs  over  all    strong  totally  geodesic   submanifolds  
 of   $M'$;
 \item
 Let
 $f: M'\longrightarrow \mathbb{C}^N$ be  
 a  complex  $(k,r)$-regular map.  
 Then 
  \begin{eqnarray}\label{eq-3.1hhh}
 N\geq   \sup\{t(k)\mid  \bar  c_{t(k)}( \boldsymbol{\xi}(M,k,-+r;\mathbb{C}))\neq 0 \}+k 
 \end{eqnarray}
 where  $M$  runs  over  all     strong  totally  geodesic   submanifolds  
 of   $M'$.  
 \end{enumerate} 
\end{theorem}

\begin{proof}
Let      $M\subseteq  M'$  be  a  strong   totally  geodesic   submanifold.

 (1)   By  Lemma~\ref{le-25jan09z},  we  have  a  $(k,r)$-regular  map 
 $f: M\longrightarrow \mathbb{R}^N$.    
   By   Theorem~\ref{25-cor1},  we  obtain  (\ref{eq-3.1zzz}).  
   
  (2)   By  Lemma~\ref{le-25jan09z},  we  have  a  complex  $(k,r)$-regular  map 
 $f: M\longrightarrow \mathbb{C}^N$.    
   By   Theorem~\ref{25-cor2},  we  obtain  (\ref{eq-3.1hhh}).  
\end{proof}

\subsection{Obstructions  for  regular  embeddings  of  graphs}
\label{ss7.2}

Let  $G=(V_G,E_G)$  be  a  graph.  A  {\it   non-vanishing  weight}  on  $G$   is  
a  function   $w:  E_G\longrightarrow  (0, +\infty)$  assigning  each  edge  with  a  positive  number.  
 Identify   any  edge $e\in  E_G$  with  a  $1$-cell  consisting  of 
 an  open  segment  of  length  $w(e)$.  
  Then  we  have  a  $1$-dimensional  CW-complex  $|G|_w$,
  which  is  homeomorphic  to  the  geometric  realization  $|G|$   of  $G$,  with  a  metric  
  $d_w:  |G|_w\times  |G|_w\longrightarrow 
  [0, +\infty]$   such  that  for  any  $x,y\in  |G|_w$,  
   $d_w(x,y) $  is  the  infimum  of   the   lengths   of  paths  in  $|G|_w$  connecting  $x$  and  $y$ 
   and  $d_w(x,y)=+\infty$  if  $x$  and  $y$  belong  to  different  path  components  of  $|G|_w$.

Let  $G$  and  $G'$  be  graphs.  
Let  $w$  and  $w'$  be    non-vanishing  weights   on  $G$  
 and  $G'$  respectively.  
 Suppose  $G$  is  a  subgraph  of  $G'$  and  $w$  is  the  restriction  of  $w'$  on  $G$.    
 Then   for  any  $x,y\in  |G|_w$,  
 \begin{eqnarray}\label{eq-250211-weight1}
 d_{w}(x,y)\geq  d_{w'}(x,y). 
 \end{eqnarray}
 The  next  definition  is  a  weighted  version  of  Definition~\ref{def-250120-9}. 
 
\begin{definition}\label{def-250120-10}  
We  say  that  
$(G,w)$  is  a  {\it   strong   totally  geodesic  weighted subgraph}
  of   $(G',w')$  if  the  equality  in  (\ref{eq-1.9})   is  satisfied  for  any  $x,y\in  |G|_w$.  
  \end{definition}

In  particular,  
if  we  let   $w'$  to  be    the  unit  weight 
 with  constant  value  $1$  on  any  edges  of  $G'$,  
then  $(G,1)$  is  a     strong   totally  geodesic  weighted subgraph 
  of   $(G',1)$  iff  $G$  is  a   strong   totally  geodesic    subgraph   of  $G'$.

\begin{lemma}\label{le-7.25jan0h}
 Let  $(G,w)$  be  a  strong  totally  geodesic  weighted   subgraph 
  of    $(G',w')$.     
   If  $f:  |G'|_{w'}\longrightarrow  \mathbb{F}^N$  is  a
  $(k,r)$-regular  map,  
  then   the  restriction  of  $f$  on  $|G|_w$  induces  a  $(k,r)$-regular  map  
  $f:  |G|_w\longrightarrow  \mathbb{F}^N$,  where  $G$  is  equipped  with 
  the  weight  $w=w'\mid_G$.   
  \end{lemma}
  
  \begin{proof}
  The  lemma  follows  from    Definition~\ref{def-250120-10}    and   the  observation  (5)  
  of   Definition~\ref{def-7.1}.   
  \end{proof}

Theorem~\ref{25-cor1} and  Theorem~\ref{25-cor2}  can  be  applied  to  
    (real  or  complex)  $(k,r)$-regular   embedding  problems  of  weighted  graphs   
with   the  distance   given  in  the  first  paragraph  of  this  subsection.

\begin{theorem} \label{25-cor2zzz} 
Let  $G'$  be  a  graph  with  a  non-vanishing  weight  $w'$.  
\begin{enumerate}[(1)]
\item

Let
 $f: |G'|_{w'}\longrightarrow \mathbb{R}^N$ be  
 a $(k,r)$-regular map.  
 Then 
  \begin{eqnarray}\label{eq-3.1zzz-g}
 N\geq   \max \{t(k)\mid  \bar w_{t(k)}( \boldsymbol{\xi}(|G|_w,k,-+r;\mathbb{R}))\neq 0 \}+k; 
 \end{eqnarray}

 \item
 Let
 $f: |G'|_{w'}\longrightarrow \mathbb{C}^N$ be  
 a  complex  $(k,r)$-regular map.  
 Then 
  \begin{eqnarray}\label{eq-3.1hhh-g}
 N\geq   \max  \{t(k)\mid  \bar  c_{t(k)}( \boldsymbol{\xi}(|G|_w,k,-+r;\mathbb{C}))\neq 0 \}+k.  
 \end{eqnarray}
 \end{enumerate} 
 Here  in  both  (\ref{eq-3.1zzz-g})  and   (\ref{eq-3.1hhh-g}),   
    $(G,w)$  runs  over  all  strong    totally  geodesic  weighted  subgraphs  
 of   $(G',w')$  such  that  $w=w'|_G$. 
\end{theorem}

\begin{proof}
Let    $(G,w)$  be   a      strong   totally  geodesic  weighted subgraph 
  of   $(G',w')$  such  that  $w=w'\mid_G$.

 (1)   By  Lemma~\ref{le-7.25jan0h},  we  have  a  $(k,r)$-regular  map 
 $f:  |G|_w\longrightarrow \mathbb{R}^N$.    
   By   Theorem~\ref{25-cor1},  we  obtain  (\ref{eq-3.1zzz-g}).  
   
  (2)   By  Lemma~\ref{le-7.25jan0h},  we  have  a  complex  $(k,r)$-regular  map 
 $f: |G|_w\longrightarrow \mathbb{C}^N$.    
   By   Theorem~\ref{25-cor2},  we  obtain  (\ref{eq-3.1hhh-g}).  
\end{proof}

Now  we  consider  integer-valued  weights  on  graphs.  

\begin{lemma}
\label{co-2501991}
Let  $w: E_G\longrightarrow  \{1,2,\ldots\}$  be  a  integer-valued  non-vanishing  weight  on  $G$.  
Let  ${\rm  sd}(G;w)$  be  the  subdivision  of  $G$   
adding  $w(e)-1$  vertices  to  each  $e\in E_G$  dividing  $e$  equally  into  $w(e)$  edges.  
Then   $|{\rm  sd}(G;w)|_1$  is   isometrically  homeomorphic  to  $|G|_w$. 
Moreover,     
\begin{eqnarray}\label{eq-25-01-w-5}
\boldsymbol{\xi}(V_{{\rm  sd}(G;w)},k,-;\mathbb{F})=j^* \boldsymbol{\xi}(|G|_w,k,-;\mathbb{F})
\end{eqnarray}
where  
\begin{eqnarray}\label{eq-25-01-w-3}
j:  {\rm  Conf}_k(V_{{\rm  sd}(G;w)}, -)/\Sigma_k\longrightarrow   {\rm  Conf}_k(|G|_w, -)/\Sigma_k 
\end{eqnarray}
is  an  embedding  induced  by  the  canonical  inclusion of   the  $0$-skeleton  
  \begin{eqnarray*}
  V_{{\rm  sd}(G;w)}\subseteq   |{\rm  sd}(G;w)|_1  \cong   |G|_w.  
  \end{eqnarray*}
\end{lemma}

\begin{proof}
Recall  that  
$V_G$  is  equipped  with  the  metric  $d_G$  introduced  in  Subsection~\ref{ss3.1}.   
On  the  other  hand,   $|G|_w$  is  equipped  with  the  metric  $d_w$  
and    $|{\rm  sd}(G;w)|_1$  is  equipped  with  the  metric  $d_1$  where $1$  is  the  constant 
weight  assigning each  edge   in  $  {\rm  sd}(G;w)$   with  the  constant  $1$.  
The  identity  map  is  an  isometry 
\begin{eqnarray}\label{eq-25-01-w-1}
{\rm  id}:  (|G|_w,  d_w)  \overset{\cong}{\longrightarrow}  (|{\rm  sd}(G;w)|_1, d_1).  
\end{eqnarray}
  The  canonical  inclusion    of   the  $0$-skeleton  
     is   an  isometric  embedding 
     \begin{eqnarray}\label{eq-25-01-w-2}
      (V_{{\rm  sd}(G;w)},d_{{\rm  sd}(G;w)}) 
       \longrightarrow    (|{\rm  sd}(G;w)|_1, d_1).
     \end{eqnarray}   
 It  follows  from (\ref{eq-25-01-w-1})  and  (\ref{eq-25-01-w-2})  that  
  (\ref{eq-25-01-w-3})  is  a  map  of  persistent  CW-complexes  induced  by  
  (\ref{eq-25-01-w-2}).  
  Consequently,  (\ref{eq-25-01-w-3})  induces  a  pull-back   (\ref{eq-25-01-w-5}) of  persistent  vector  bundles.  
\end{proof}

The  next  theorem  follows  from  Theorem~\ref{25-cor2zzz} 
and  Lemma~\ref{co-2501991}.  

\begin{theorem}\label{co-25-01-graph}
Let  $G'$  be  a  graph  with  an  integer-valued  non-vanishing  weight  $w'$.  
\begin{enumerate}[(1)]
\item

Let
 $f: |G'|_{w'}\longrightarrow \mathbb{R}^N$ be  
 a $(k,r)$-regular map.  
 Then 
  \begin{eqnarray}\label{eq-3.1zzz99}
 N\geq  
  \max \{t(k)\mid  \bar w_{t(k)}( \boldsymbol{\xi}(V_{{\rm  sd}(G;w)},k,-+r;\mathbb{R}))\neq 0 \}+k; 
 \end{eqnarray}
 \item
 Let
 $f: |G'|\longrightarrow \mathbb{C}^N$ be  
 a  complex  $(k,r)$-regular map.  
 Then 
  \begin{eqnarray}\label{eq-3.1hhh99}
 N\geq   
 \max  \{t(k)\mid  \bar  c_{t(k)}( \boldsymbol{\xi}(V_{{\rm  sd}(G;w)},k,-+r;\mathbb{C}))\neq 0 \}+k.   
 \end{eqnarray}
 \end{enumerate} 
  Here   in  both  (\ref{eq-3.1zzz99})  and   (\ref{eq-3.1hhh99}),   
    $(G,w)$  runs  over  all  strong    totally  geodesic  weighted  subgraphs  
 of   $(G',w')$  such  that  $w=w'|_G$. 
\end{theorem}

\begin{proof}
  By  (\ref{eq-25-01-w-5}), 
\begin{eqnarray*}
  \bar w_{t(k)}( \boldsymbol{\xi}(V_{{\rm  sd}(G;w)},k,-+r;\mathbb{R}))\neq  0 
   \Longrightarrow  
 \bar w_{t(k)}( \boldsymbol{\xi}(|G|_w,k,-+r;\mathbb{R}))\neq 0. 
\end{eqnarray*}
Thus 
\begin{eqnarray}
 && \max \{t(k)\mid  \bar w_{t(k)}( \boldsymbol{\xi}(|G|_w,k,-+r;\mathbb{R}))\neq 0 \}
 \nonumber\\
 &\geq&
 \max  \{t(k)\mid  \bar w_{t(k)}( \boldsymbol{\xi}(V_{{\rm  sd}(G;w)},k,-+r;\mathbb{R}))\neq 0 \}.     
 \label{eq-inequal1}
\end{eqnarray}
Therefore,  
(\ref{eq-3.1zzz99})  follows  from  (\ref{eq-3.1zzz-g})   and  (\ref{eq-inequal1}).  
 Similarly,  by  (\ref{eq-25-01-w-5}), 
\begin{eqnarray*}
  \bar  c_{t(k)}( \boldsymbol{\xi}(V_{{\rm  sd}(G;w)},k,-+r;\mathbb{C}))\neq  0 
   \Longrightarrow  
 \bar  c_{t(k)}( \boldsymbol{\xi}(|G|_w,k,-+r;\mathbb{C}))\neq 0. 
\end{eqnarray*}
Thus 
\begin{eqnarray}
 && \max  \{t(k)\mid  \bar  c_{t(k)}( \boldsymbol{\xi}(|G|_w,k,-+r;\mathbb{C}))\neq 0 \}
 \nonumber\\
 &\geq&
 \max  \{t(k)\mid  \bar  c_{t(k)}( \boldsymbol{\xi}(V_{{\rm  sd}(G;w)},k,-+r;\mathbb{C}))\neq 0 \}.     
 \label{eq-inequal2}
\end{eqnarray}
Therefore,  
(\ref{eq-3.1hhh99})  follows  from  (\ref{eq-3.1hhh-g})   and  (\ref{eq-inequal2}).  
\end{proof}

\section{Applications  of  the  regular embeddings  to  geometric  realizations  of  
the  independence  complexes }
\label{s-777}

 The  geometric  realizations  of  simplicial  complexes  attract  attention    in   
 computer  science  (cf.  \cite{geom-real}).  
Let  $V$  be  a  discrete  set  with a  total  order  $\prec$.  
An  {\it  abstract  simplicial  complex}  $\mathcal{K}$  with  vertices  from  $V$  
 is  a  family  of   nonempty  subsets  of  $V$  such  that 
 for   any  $\sigma\in  \mathcal{K}$  and  any  nonempty  subsets  $\tau$  of  $\sigma$,  
 it holds  $\tau\in \mathcal{K}$  (cf.  \cite[p.  107]{at}).  
 An  element     $\sigma\in \mathcal{K}$  is  called  a  {\it  simplex}    (cf.  \cite[p.  103]{at}). 
 Let $\sigma=\{v_0,v_1,\ldots,v_n\}$    be   a  simplex
   where  $v_0\prec v_1\prec \cdots\prec  v_n$.
  A  {\it  geometric  realization}  of  $\sigma$ 
  is  a  map  $i_\sigma:  \sigma\longrightarrow \mathbb{R}^{n+1}$  such  that  
  $i_\sigma(v_k)= e_k$ for  $0\leq  k\leq  n$,  where  
  $e_0$,  $e_1$,  $\ldots$,  $e_n$  is  a   basis  of  $\mathbb{R}^{n+1}$. 
  Given    a  simplicial  complex  $\mathcal{K}$,
  let  $V_{\mathcal{K}}=\bigcup_{\sigma\in \mathcal{K}}  \sigma$  
 be  the  set  of  vertices  of $\mathcal{K}$.    
 A  {\it  geometric  realization}  of     $\mathcal{K}$   
 is   a  map  $i_\mathcal{K}:   V_{\mathcal{K}}\longrightarrow  \mathbb{R}^{N+1}$ 
  such  that  for  any  $\sigma\in \mathcal{K}$,  
  the  restriction  
  $i_\sigma=  i_\mathcal{K}\mid_{\sigma}$  induces  
 a  geometric  realization  of  $\sigma$ in  an  affine  subspace  of    $\mathbb{R}^N$.

 Let  $(X,d)$  be  a  metric  space. 
 Consider  the  {\it  persistent  independence    complex}  
 \begin{eqnarray}\label{eq-7.oqa}
   {\rm  Ind}(X,-)=\bigcup_{k\geq   1}{\rm  Conf}_{k }(X,-)/\Sigma_k  
 \end{eqnarray}
 where  the  parameter  is  $r\in  [0, +\infty]$  
 such  that  (\ref{eq-7.oqa})  is  a  filtered  simplicial  complex.  
  In  particular,  
  \begin{enumerate}[(1)]
  \item
  if  $X$  is  a  Riemannian  manifold     with   a  Riemannian  metric  $g$,    
  then    ${\rm  Ind}(X,-)$  is  a  persistent  $\Delta$-manifold  with  the  induced  
  Riemannian  metric  $g^k/\Sigma_k$; 
  \item
  if  $X$  is   the  vertex  set     of  a  graph,  then  ${\rm  Ind}(X,-)$  is  a  persistent  simplicial  complex 
  such  that  the  vertices  are  discrete.    
  \end{enumerate}

   The  $k$-regularity    can be  weakened  to 
  the  affine  $k$-regularity (cf.  \cite[Definition~2.4]{high1}).     
   We  call  a   map  $f:  X\longrightarrow  \mathbb{R}^N$   
   {\it  affine   $(k,r)$-regular}  if  for  any  distinct  
   $k$-points  $x_1,\ldots, x_k\in  X$  such  that  $d(x_i,x_j)>2r$  for  any  $i\neq  j$,  
  their  images  $f(x_1)$,  $\ldots$,  $f(x_k)$  are affinely  independent in  $\mathbb{R}^N$.   
     Note  that  the  affine $(k,0)$-regularity  implies  the  affine  $(l,0)$-regularities   for any  
     $2\leq  l\leq  k$;  and  the  affine  $(k,r)$-regularity  implies the  affine  $(k,s)$-regularity  for  
     any  $s\geq  r$.  
  
  \begin{proposition}\label{pr-7.ind-reg1}
  Suppose  an  embedding   $f:  X\longrightarrow  \mathbb{R}^N$   is  
  affine  $(l,r)$-regular  for  any  $2\leq  l\leq  k$.  
  Then  $f$    
induces  a  persistent  geometric  realization  of  
the  $(k-1)$-skeleton  of  ${\rm  Ind}(X,-+r)$  in  $\mathbb{R}^N$.   
 \end{proposition}
 
 \begin{proof}
 Let  $s\geq  r$.  
 Then 
  $f:  X\longrightarrow  \mathbb{R}^N$   is  
  affine  $(l,s)$-regular  for  any  $2\leq  l\leq  k$.    
   Let  $\sigma\in  {\rm  sk}^{k-1}({\rm  Ind}(X,s))$.  
   Then  the  number  of  vertices  in  $\sigma$  is  smaller  than  or  equal to  $k$  
   and   $d(u,v)>2s$   for  any  distinct  vertices  $u$  and  $v$  in  $\sigma$.  
   Thus  $f(\sigma)=\{f(v)\mid  v\in\sigma\}$,  i.e.  
       the  images  of the  vertices  in  $\sigma$,  
   is   affinely   independent  in  $\mathbb{R}^N$.  
   Hence  $f\mid _\sigma $  is  a  geometric  realization  of  $\sigma$.  
   Therefore,  $f$  induces   a  geometric  realization  of   ${\rm  sk}^{k-1}({\rm  Ind}(X,s))$ 
   in  $\mathbb{R}^N$.   Taking  the  persistence,  we  obtain  a  
   persistent  geometric  realization  of  
${\rm  sk}^{k-1}({\rm  Ind}(X,-+r))$   in  $\mathbb{R}^N$. 
 \end{proof}
 
 \begin{corollary}\label{pr-7.ind-reg}
Suppose  an  embedding   $f:  X\longrightarrow  \mathbb{R}^N$   is  
  $(l,r)$-regular  for  any  $2\leq  l\leq  k$.  
  Then  $f$    
induces  a  persistent  geometric  realization  of  
the  $(k-1)$-skeleton  of  ${\rm  Ind}(X,-+r)$  in  $\mathbb{R}^N$.   
 \end{corollary}
 
   \begin{proof}
The  linear   independence  of  $l$-vectors    implies  the affine   independence  of  these  vectors.  
Thus  the $(l,r)$-regularity  of a  map  
$f: X\longrightarrow\mathbb{R}^N$ implies  the  affine  $(l,r)$-regularity  of  $f$.  
The  proof   follows  from  Proposition~\ref{pr-7.ind-reg1}.     
   \end{proof}


    \bigskip

Shiquan Ren

Address:
School  of  Mathematics and Statistics,  Henan University,  Kaifeng   475004,  China.

e-mail:  renshiquan@henu.edu.cn

  \end{document}